\documentclass[11pt]{amsart} 
\usepackage{graphicx,stackengine,scalerel}
\def\tang{\ThisStyle{\abovebaseline[0pt]{\scalebox{-1}{$\SavedStyle\perp$}}}}
\usepackage{verbatim, latexsym, amssymb, amsmath,color}
\usepackage{epsfig}
\usepackage{CJK}
\usepackage{amsmath}
\usepackage{amssymb}
\usepackage{amsthm}
\usepackage[utf8]{inputenc}
\usepackage[english]{babel}

\newtheorem{theorem}{Theorem}[section]

\newtheorem{proposition}[theorem]{Proposition}

\newtheorem{lemma}[theorem]{Lemma}
\newtheorem{claim}[theorem]{Claim}
\newtheorem{definition}[theorem]{Definition}
\newtheorem{remark}[theorem]{Remark}
\newcommand{\at}[2][]{#1|_{#2}}
\usepackage{geometry}
\title{Morse Index Bound for Minimal Torus }
\author{Yuchin Sun}
\begin{document}
	\maketitle
	\begin{abstract}
The min-max construction of minimal spheres using harmonic replacement is introduced by Colding and Minicozzi \cite{CD} and generalized by Zhou \cite{minimaltorus} to conformal harmonic torus. We prove that the Morse index of the min-max conformal harmonic torus is bounded by one.
	\end{abstract}
\section{Introduction}

Morse theory studies the topology of a manifold by analysing proper functions defined on it. Finite-dimensional Morse theory was developed by Morse \cite{Morse} to study geodesics. A natural extension of it would be a Morse theory for harmonic surfaces. Colding and Minicozzi proved the existence for min-max harmonic spheres whose sum of areas realizes the width \cite{CD}, the Morse index of such min-max harmonic sphere is bounded by one \cite{YS}. The approach by Colding and Minicozzi \cite{CD} was extended by Zhou to min-max minimal torus \cite{minimaltorus} and min-max minimal surface of genus $g\geq 2$  \cite{minmaxgenus}. We prove that the Morse index of such min-max tori is bounded by one.
\begin{theorem}\label{main}
Let $M$ be a closed Riemannian manifold of dimension at least three. For a generic choice of metric on $M$ and any homotopically nontrivial path $\beta\in\Lambda$, if the corresponding width $W$ is positive. Then there are possibly a conformal torus $v_0:\mathcal{T}_{\tau_0}\to M$ and finitely many harmonic spheres $v_j:S^2\to M$ such that 
\begin{enumerate}
\item $\text{Area}(v_0)+\sum_j\text{Area}(v_j)=W,$
\item $\text{Index}(v_0)+\sum_{j} Index(v_j)\leq 1.$
	\end{enumerate}
\end{theorem}
Morse index of min-max minimal surface is studied under many different settings. We list a few in the following and compare those result with Theorem \ref{main}. Sacks and Uhlenbeck introduced $\alpha$-energy \cite{SU}, which can be perturbed to be a Morse function. The corresponding min-max critical points $f_\alpha$ of $\alpha$-energy converge to a harmonic sphere $f_\infty$ as $\alpha\to 1$. The Morse index of the harmonic sphere $f_\infty$ is bounded by the limit inferior of Morse index of $f_\alpha$ as $\alpha\to 1$, this result serves as a key step in the proof of the sphere theorem \cite{MD}. However, the harmonic sphere $f_\infty$ fails to realize the energy identity in general. An alternative to $\alpha$-energy of Sacks and Uhlenbeck is done by Rivi\`{e}re using viscosity method \cite{Viscosity}. The parametric approach produces
immersed (rather than embedded) minimal surfaces with possible branch points for arbitrary codimension. The Morse index bound of such immersed minimal surfaces is shown in \cite{VPT}. We like to remark here that the Morse index in \cite{VPT} is for the area functional, while the Morse index in Theorem \ref{main} is for the energy functional. For a conformal harmonic map, its Morse index for area is bounded by its Morse index for energy plus a constant depending on the genus and branch points \cite{Mario}. Zhou and Cheng \cite{cheng2021existence} also use the idea of perturbed energy and generalize the
existence theory of branched immersed minimal 2-spheres by Sacks-Uhlenbeck \cite{SU} to the CMC setting in Riemannian 3-spheres. They obtain the Morse index bound by showing the Morse index bound passes through the limit as the perturbed Dirichlet energy approaches the normal energy. On the other hand, Almgren-Pitts min-max theory gives the existence of embedded minimal hypersurfaces in a closed manifold of dimension between three and seven \cite{AP}. The Morse index bound of such min-max hypersurfaces is given by Marques and Neves \cite{MN}, whom later show that there exist such min-max hypersurfaces whose Morse index is the number of parameters  (dimension of the families used in the min-max process), under the assumption of multiplicity one \cite{MN2}. The multiplicity one conjecture is proved by Zhou \cite{Mul1} and states that for a generic metric on a closed manifold $M^n$, $3\leq n\leq 7$, there exists an embedded minimal hypersurface of multiplicity one. The conjecture is also shown to be true under the Allen-Cahn setting by Chodosh and Mantoulidis \cite{CC} in dimension three. 

Theorem \ref{main} is a generalization of the author's previous work \cite{YS} of proving the Morse index bound for the min-max harmonic sphere. The novelty of Theorem \ref{main} lies in two points. First, while the sphere has a unique conformal structure, the one parameter family of mappings from tori has varying conformal structures. The technical step of proving Theorem \ref{main} is constructing the vector fields corresponding to the varying conformal structures, such that we can perturb the one parameter family of mappings away from those minimal tori with Morse index larger than one. Second, we show that the space of conformal harmonic tori in a closed manifold of dimension at least three is countable. Unlike the compactness result for the space of embedded minimal hypersurfaces \cite{BS}, we don't have codimension one condition one requirement for the space of conformal harmonic tori. Moreover, unlike the harmonic sphere which has a unique conformal structure. The Teichm\"{u}ller space for tori is uncountable. We overcome this difficulty by using the bumpy metric theorem for immersed minimal surfaces by \cite{BW} and for conformal harmonic maps by \cite{JD}.

The organization of the paper is as follows. In section 2, we give the basic definition of harmonic map, Morse index, and bubble convergence for varying conformal structures. In section 3, we prove a technical lemma \ref{unstable}. In section 4, we state Zhou's existence result \cite{minimaltorus} in the form of bubble convergence. In section 5, we prove the main result Theorem \ref{main}.

\section{Preliminaries}
\subsection{Harmonic Maps}
Let $(\Sigma,h)$ be a Riemann surface with a metric $h=\lambda^2dzd\Bar{z}$ in conformal coordiantes $z=x+\sqrt{-1}y$. Let $(M,g)$ be a closed Riemannian manifold of dimension $d$, and let its metric in local coordinated be given by $g_{ij}$, with Christoffel symbols $\Gamma^i_{kl}.$

For $u\in W^{1,2}(\Sigma,M)$, the energy of $u$ on $\Sigma$ is
\begin{equation}\label{energy}
E(u,\Sigma)=\frac{1}{2}\int_{\Sigma}g_{ij}(u^i_xu^j_x+u^i_yu^j_y)dxdy.    
\end{equation}
A solution of the corresponding Euler-Lagrange equations
\begin{equation}\label{EL}
    \Delta u^i+\Gamma^i_{kl}(u^k_xu^l_x+u^k_yu^l_y)=0,\quad i=1,...,d,
\end{equation}
is called a harmonic map. Note that $\eqref{energy}$ and $\eqref{EL}$ are conformally invariant.
If we isometrically embed $M$ into some Euclidean space $\mathbb{R}^N$, then \eqref{EL} can be written as the following:
\begin{equation}\label{EL2}
    \Delta u+A(u)(\nabla u,\nabla u)=0,
\end{equation}
where $A(\cdot,\cdot)$ is the second fundamental form of $M$ in $\mathbb{R}^N$. Any $u\in W^{1,2}(\Sigma,M)$ that satisfies \eqref{EL2} weakly is smooth.
We introduce  \textit{nearest point projection} $\Pi:\mathbb{R}^N\to M$ which maps a point $x\in\mathbb{R}^N$ to the nearest point of $M$. There is a tubular neighborhood of $M$, $M_\delta=\big\{x\in\mathbb{R}^N:\text{ dist}(x,M)<\delta\big\},$ on which $\Pi$ is well-defined and smooth. For a given $X\in C^{\infty}(\Sigma,\mathbb{R}^N)$, we consider the variation of $u$ with respect to $X$ defined as the following:
\begin{equation}\label{Variation}
	\Pi\circ(u+sX),\quad s\in\mathbb{R}.
\end{equation}
$\Pi\circ(u+sX)$ is well-defined for $s$ sufficiently small such that $(u+sX)(\Sigma)\subset M_\delta$. The Taylor polynomial expansion of $\Pi$ is the following:
\begin{equation}\label{Taylor}
\Pi\circ(u+sX)=u+sd\Pi_u(X)+\frac{s^2}{2}\text{Hess}\Pi_u(X,X)+o(s^3).
\end{equation}
From \cite[2.12.3]{LS}, we know that the following properties of the nearest point projection $\Pi$ hold:
\begin{equation}\label{proj}
d\Pi_u(X)=p_u(X),    
\end{equation}
where $p_u$ denotes orthogonal projection of $\mathbb{R}^N$ onto $T_uM$,
\begin{equation}\label{proj1}
v_1\cdot\text{Hess}\Pi_u(v_2,v_3)=\frac{1}{2}\sum v_{\sigma_1}^\perp\cdot\text{Hess}\Pi_u(v^{\tang}_{\sigma_2},v^{\tang}_{\sigma_3}),\quad v_1,v_2,v_3\in\mathbb{R}^N,   
\end{equation}
where $v^\perp\at[]{u}=p_u(v)$, $v^{\tang}=v-v^\perp$; the sum on the right is over all $6$ permutations $\sigma_1,\sigma_2,\sigma_3$ of $1,2,3$ and $\text{Hess}\Pi_u$ denotes the Hessian of $\Pi$ at $u$, 
\begin{equation}\label{proj2}
  \text{Hess}\Pi_u(v_1,v_2)=-A_u(v_1,v_2),\quad v_1,v_2\in T_uM,
\end{equation}
where $A_u$ is the second fundamental form of $M$ at $u$. By applying $\nabla$ to \eqref{Taylor} we have
\begin{equation}
\begin{split}
\nabla\Big(\Pi\circ(u+sX)\Big)=&\nabla u+s(d\Pi_u(\nabla X)+\text{Hess}\Pi_u(X,\nabla u))\\
	&+\frac{s^2}{2}\big(2\text{Hess}\Pi_u(X,\nabla X)+\nabla^3\Pi_u(X,X,\nabla u)\big)+o(s^3),
\end{split}
\end{equation}
and the energy of $\Pi\circ(u+sX)$ is
\begin{equation}\label{energyTaylor}
    \begin{split}
        E(\Pi\circ(u+sX),\Sigma)=&\int_\Sigma\langle\nabla\Big(\Pi\circ(u+sX)\Big),\nabla\Big(\Pi\circ(u+sX)\Big)\rangle\\
        =&\int_{\Sigma}\langle\nabla u,\nabla u\rangle \\
        +2s&\int_{\Sigma}\langle\nabla u,d\Pi_u(\nabla X)\rangle+\langle\nabla u,
        \text{Hess}\Pi_u(X,\nabla u)\rangle \\
        +s^2&\Big\{\int_{\Sigma}\langle d\Pi_u(\nabla X),d\Pi_u(\nabla X)\rangle \\
	&+\int_{\Sigma}2\langle \text{Hess}\Pi_u(X,\nabla u),d\Pi_u(\nabla X)\rangle \\
	&+\int_{\Sigma}\langle \text{Hess}\Pi_u(X,\nabla u),\text{Hess}\Pi_u(X,\nabla u)\rangle \\
	&+\frac{1}{2}\int_{\Sigma}\langle\nabla u,\big( 2\text{Hess}\Pi_u(X,\nabla X)+\nabla^3\Pi_u(X,X,\nabla u)\big)\rangle \Big\}\\
        +&o(s^3).
    \end{split}
\end{equation}
Then we have
\begin{equation}\label{firstvariation}
    \begin{split}
        \frac{d}{ds}E(\Pi\circ(u+sX),\Sigma)=&\frac{d}{ds}\int_\Sigma\langle\nabla\Big(\Pi\circ(u+sX)\Big),\nabla\Big(\Pi\circ(u+sX)\Big)\rangle\\
        =&2\int_{\Sigma}\langle\nabla u,d\Pi_u(\nabla X)\rangle+\langle\nabla u,
        \text{Hess}\Pi_u(X,\nabla u)\rangle+o(s).
    \end{split}
\end{equation}
For such a variation we have that the equation $\frac{d}{ds}\at[\Big]{s=0} E(\Pi\circ(u+sX),\Sigma)=0$, and using \eqref{proj}, \eqref{proj1}, \eqref{proj2} gives the following integral identity
\begin{equation}
    \int_\Sigma\langle X, \Delta u+A_u(\nabla u,\nabla u)\rangle=0,
\end{equation}
which coincides with \eqref{EL2}. Now we look at second variation of energy:
\begin{equation}\label{energyTaylor2}
    \begin{split}
        \frac{d^2}{ds^2}E(\Pi\circ(u+sX),\Sigma)=&2\Big\{\int_{\Sigma}\langle d\Pi_u(\nabla X),d\Pi_u(\nabla X)\rangle \\
	&+\int_{\Sigma}2\langle \text{Hess}\Pi_u(X,\nabla u),d\Pi_u(\nabla X)\rangle \\
	&+\int_{\Sigma}\langle \text{Hess}\Pi_u(X,\nabla u),\text{Hess}\Pi_u(X,\nabla u)\rangle \\
	&+\int_{\Sigma}\langle\nabla u,\big( 2\text{Hess}\Pi_u(X,\nabla X)+\nabla^3\Pi_u(X,X,\nabla u)\big)\rangle \Big\}\\
	&+o(s).
    \end{split}
\end{equation}
If $u$ is harmonic and assume that $X(x)\in T_{u(x)}M$, $\forall x\in\Sigma$. By \eqref{EL2}, \eqref{proj1}, \eqref{proj2} we have the following 
\begin{equation}
    \langle \text{Hess}\Pi_u(X,\nabla u),d\Pi_u(\nabla X)\rangle=0,
\end{equation}
\begin{equation}
    \langle \text{Hess}\Pi_u(X,\nabla u),\text{Hess}\Pi_u(X,\nabla u)\rangle=\langle A(X,\nabla u),A(X,\nabla u)\rangle,
\end{equation}
\begin{equation}
    \begin{split}
\int_\Sigma\langle\nabla u,\big( 2\text{Hess}\Pi_u(X,\nabla X)+\nabla^3\Pi_u(X,X,\nabla u)\big)\rangle&=\int_\Sigma\langle\nabla u,\nabla(\text{Hess}\Pi_u(X,X))\rangle\\
&=-\int_\Sigma\langle\Delta u,\text{Hess}\Pi_u(X,X)\rangle\\
&=\int_\Sigma\langle A(\nabla u,\nabla u), \text{Hess}\Pi_u(X,X)\rangle\\
&=-\int_\Sigma\langle A(\nabla u,\nabla u),A(X,X)\rangle.
\end{split}
\end{equation}
Then the second variation of a harmonic map $u$ is 
\begin{equation}
\begin{split}
    \frac{d^2}{ds^2}\at[\Big]{s=0}E(\Pi\circ(u+sX),\Sigma)=&\int_\Sigma\langle d\Pi_u(\nabla X),d\Pi_u(\nabla X)\rangle\\
    -&\int_\Sigma\langle A(\nabla u,\nabla u),A(X,X)\rangle-\langle A(X,\nabla u),A(X,\nabla u)\rangle\\
    =&\int_\Sigma\langle\nabla X,\nabla X\rangle-\langle R^M(\nabla u,X)X,\nabla u\rangle,
\end{split}
\end{equation}
where $R^M(\cdot,\cdot)$ is the Riemann curvature tensor of $M$.
\begin{definition}[Index Form]
The index form of a harmonic map $u:\Sigma\to M$ is defined as the following
\begin{equation}\label{index form}
    I(X,Y)=\int_{\Sigma}\langle\nabla X,\nabla Y \rangle-\int_{\Sigma} \langle R^M(\nabla u,X)Y,\nabla u\rangle,\quad X,Y\in\Gamma(u^{-1}TM),
\end{equation}
where $u^{-1}TM$ is a bundle over $\Sigma$ with metric $g_{ij}(u(x))$.
\end{definition}
\begin{definition}[Index]\label{indexdef} The index of a harmonic map $u:\Sigma\to M$ is the maximal dimension of the subspace $X$ of $\Gamma(u^{-1}TM)$ on which the index form is negative definite.
\end{definition}
\subsection{Convergence Result for Harmonic Torus}
Any flat torus $T^2$ can be viewed as the quotient space of $\mathbb{C}$ moduled by a lattice generated by bases $\{\omega_1,\omega_2\}$. After some conformal linear transformation, we can assume $\omega_1=1$ and $\omega_2=\tau\in\mathbb{H}$. Let $T_\tau$ be the torus by gluing edges of the lattice $\{1,\tau\}$ with the plan metric $dzd\bar{z}$. Denote $\iota_{\tau_1,\tau_2}$ to be the diffeomorphism from $\mathcal{T}_{\tau_1}$ to $\mathcal{T}_{\tau_2}$, which is the quotient map of the linear map of $\mathbb{C}$ keeping $1$ and sending $\tau_1$ to $\tau_2$. We denote $T_0$ by $T_{\sqrt{-1}}$ and $\iota_{\tau}$ to be the linear map of $\mathbb{C}$ keeping $1$ and sending $\sqrt{-1}$ to $\tau$. For a given manifold $M$, we denote $(u,T_\tau)$ by the map $u:T_\tau\to M$ with the flat torus $T_\tau$ as its domain. \begin{proposition}
\cite[Proposition 3.1]{minimaltorus}
Let $g$ be a $C^1$ metric on $T_0$. There exists a unique mark $\tau\in\mathbb{H}$, and a unique orientation preserving $C^{1,\frac{1}{2}}$ conformal diffeomorphism $h:T_\tau\to (T_0,g)$
, such that
$h$ is isotopic to $(\iota_\tau)^{-1}$. \end{proposition}

To discuss bubble convergence for a sequence of maps $\{(u_i,\mathcal{T}_{\tau_i})\}_{i\in\mathbb{N}}$, we need to consider the convergence of the metrics given by $\mathcal{T}_{\tau_i}$. In fact, metrics of $T_\tau$ and $\mathcal{T}_{\tau'}$ separately are conformally equivalent if they lie in the same orbit of PSL($2,\mathbb{Z}$). Denote
\[\mathcal{M}=\{z\in\mathbb{C},|z|\geq 1, \text{Im}z>0,-\frac{1}{2}<\text{Re}z<\frac{1}{2},\text{ and if }|z|=1,\text{Re}z\geq 0\}\]
to be fundamental domain for $\text{PSL}(2,\mathbb{Z})$ (which is also equivalent to the quotient space $\mathbb{H}/\text{PSL}(2,\mathbb{Z})$). Since every metric of some $T_\tau$ is conformally equivalent to an element in $\mathcal{M}$ after a PSL$(2,\mathbb{Z})$ action. We say a sequence $\{\tau_i\}$ converges to $\tau_\infty\in\mathcal{M}$ if after being conformally translated to $\{\tau'_i\}\subset\mathcal{M}$ by actions in PSL$(2,\mathbb{Z})$, $\tau_i'\to\tau_\infty$. 
\begin{theorem}\cite[Theorem 5.1]{minimaltorus}\label{torusconvergence}
Let $\{(u_i,\mathcal{T}_{\tau_i})\}_{i\in\mathbb{N}}$, $u_i:\mathcal{T}_{\tau_i}\to M$, be a sequence of maps such that
\begin{enumerate}
    \item $u_i\in W^{1,2}(\mathcal{T}_{\tau_i},M)$, $\forall i$,
    \item $E(u_i,\mathcal{T}_{\tau_i})\leq \bar{E}$, $\forall i$, for some $\bar{E}>0$,
    \item $E(u_i,\mathcal{T}_{\tau_i}))-\text{Area}(u_i)\to 0$, as $i\to\infty,$
    \item For any finite collection of disjoint balls $\bigcup_i B_i$ on $\mathcal{T}_{\tau_i}$ such that $E(u_i,\bigcup_iB_i)\leq\epsilon_0$ (for the choice of $\epsilon_0$, see \cite{minimaltorus}), let $v$ be the harmonic replacement of $u_i$ on $\frac{1}{8}\bigcup_iB_i$. We have 
    \[\int_{\frac{1}{8}\bigcup_iB_i}|\nabla u_i-\nabla v|^2\to 0,\quad i\to\infty.\]
\end{enumerate}
Then either of the following holds
\begin{enumerate}
    \item If $\tau_i\to\tau_\infty$ in the above sense, then there exists a conformal harmonic torus $v_0:\mathcal{T}_{\tau_\infty}\to M$ and finitely many harmonic spheres $v_j:S^2\to M$ such that $u_i$ bubble converge to up to subsequence, with
    \begin{equation}\label{energyidentity}
        \lim_{i\to\infty}E(u_i,\mathcal{T}_{\tau_i})=E(v_0,\mathcal{T}_{\tau_\infty})+\sum_jE(v_j).    
    \end{equation}
    \item If $\tau_i$ diverge, then there exists finitely many harmonic spheres $\{v_j\}$, such that $u_i$ bubble converge to up to subsequence, with body map degenerated, and
    \begin{equation}
        \lim_{i\to\infty}E(u_i,\mathcal{T}_{\tau_i})=\sum_jE(v_j).
    \end{equation}
\end{enumerate}
\end{theorem}
\begin{proof}[Sketch of proof]
Let $\{\tau'_i\}\subset\mathcal{M}$ be the corresponding conformal structure of $\mathcal{T}_{\tau_i}$ after PSL$(2,\mathbb{Z})$ actions.
\begin{description}
\item[Case 1: $\tau'_i\to\tau_\infty$]
Since $\tau'_i$ converge, we can identify a point $x\in \mathcal{T}_{\tau_\infty}$ as on $\mathcal{T}_{\tau'_i}$ by viewing it as on the fundamental regions of lattices $\{1,\tau_\infty\}$ and $\{1,\tau'_i\}$ of corresponding conformal structures. For any $x\in \mathcal{T}_{\tau_\infty}$, we consider a sequence of energy concentration radii $r_i(x)$ defined as following
\[r_i(x)=\sup\Big\{r>0, E(u_i,B_r(x))\leq\epsilon_1\Big\},\]
where $\epsilon_1<\epsilon_0$ (see \cite{minimaltorus} for the choice of $\epsilon_0$). Such $r_i(x)$ exist and are positive. We say $x$ is an \textbf{energy concentration point} if $\lim_{i\to\infty}r_i(x)=0$. If $x$ is an energy concentration point. We have that
\[\inf_{r>0}\Big\{\lim_{n\to\infty}E(u_i,B_r(x))\Big\}\geq\epsilon_1.\]
Since the sequence $\{(u_i,\tau_i)\}$ has uniform bounded energy. We know that the number of the energy concentration points are bounded by $W/\epsilon_1$. Denote the energy concentration points by $\{x_1,...,x_m\}.$ If $x\in \mathcal{T}_{\tau_\infty}\setminus\{x_1,...,x_m\}$. We can find $r(x)>0$ such that $E(u_i,B_{r(x)}(x))\leq\epsilon_1$, $\forall i$. By assumption (4), there exists $v_i$, which are the energy minimizing harmonic maps defined on $\frac{1}{8}B_{r(x)}(x)$ with the same boundary value as $u_i$, we know from \cite{SU} that $v_i$ have uniform interior $C^{2,\alpha}$-estimates on $\frac{1}{8}B_{r(x)}(x)$ hence converges to a harmonic map $v_0$ on $\frac{1}{9}B_{r(x)}(x)$ up to subsequence. For any compact subset $K\subset \mathcal{T}_{\tau_\infty}\setminus\{x_1,...,x_m\}$, we can cover them by finitely many balls $\frac{1}{9}B_{r(x)}(x)$ and hence $u_i\to v_0$ in $W^{1,2}(K,M)$ up to subsequence. After exhausting $\mathcal{T}_{\tau_\infty}\setminus\{x_1,...,x_m\}$ by a sequence of compact sets and a diagonal argument, we know $v_0$ is a harmonic map on $\mathcal{T}_{\tau_\infty}\setminus\{x_1,...,x_m\}$ and by removable singularity property of harmonic map \cite{SU} we know $v_0$ extends to a harmonic map on $\mathcal{T}_{\tau_\infty}$.

We now see what happens near the energy concentration points. Fix an energy concentration point $x\in\{x_1,...,x_m\}$ and denote $r_i=r_i(x)$. Find $r>0$ such that $E(v_0,B_r(x))\leq\frac{1}{3}\epsilon_1$. We recale $u_i$ on $r_i$. Define 
\[u_i'(y):=u_i(x+r_i(y-x)).\]
So $B_{r_i}(x)$ are rescaled to $B_1$. $u_i'$ can be viewed as defined on $B_{r/r_i}$ with $r/r_i\to\infty$. Since $B_{r/r_i}\to\mathbb{C}$, and $\mathbb{C}$ is conformally equivalent to $S^2$ without the south pole, we can consider $u_i'$ as defined on any compact subset of $S^2$ away from south pole for $i$ sufficiently large enough. Since the assumption (4) is conformally invariant, we can repeat the same argument, find finitely many energy concentration points $\{x_1',...x_n'\}\subset S^2\setminus\{\text{south pole}\}$, such that $u_i'$ converges to a harmonic map $v_1$ defined on $S^2$ in the above sense. 

We can repeat the bubbling convergence given above. There are only finitely many such steps, and then the bubbling convergence stops. For the energy identity \eqref{energyidentity}, we need to show that the following is true:
\begin{equation}\label{energyidentity2}
\lim_{r\to 0,R\to\infty}\lim_{i\to\infty}E(u_i,B_r(x)\setminus B_{r_iR}(x))=0.    
\end{equation}
\eqref{energyidentity2} follows from \cite[Proposition 5.1]{minimaltorus}.
\item[Case 2: $\tau'_i$ diverge]
We use $(t,\theta)$ as parameters on $\mathcal{T}_{\tau_i}$. Let $\theta_n=\text{arg}(\tau_n)$, and let $z'=t+\sqrt{-1}\theta=e^{-\sqrt{-1}(\frac{\pi}{2}-\theta_n)}z$ be another conformal parameter system on $\mathcal{T}_{\tau_i}$. We can argue as in case 1 to show that $u_i$ converges to a harmonic map $v_1$ defined on $S\times\mathbb{R}$. We know that $v_1$ must be nontrivial or else it contradicts with \cite[Proposition 5.1]{minimaltorus}. As $S\times\mathbb{R}$ is conformally equivalent to $S^2\setminus$\{north pole, south pole\}, we can extend $v_1$ to a harmonic map on $S^2$. For the energy concentration points, we can rescale $u_i$ and the rescaled map will converge as we discussed in case 1 to finitely many bubble maps $\{v_j\}$.
\end{description}
\end{proof}
Inspired by Theorem \ref{torusconvergence}, we come up with the following definition of \textit{bubble norm} which is used to describe how \text{close} a given map is to a collection of a conformal harmonic torus and finitely many harmonic spheres in the bubble tree sense.
\begin{definition}[Bubble Norm]\label{bubblenorm}
From Theorem \ref{torusconvergence}, we know that if a sequence of maps $\{(u_i,\tau_i)\}$ with varying conformal structures of the domains $\mathcal{T}_{\tau_i}$ converges in the bubble tree sense, the domain of the body map could be either torus or sphere depends on whether the conformal structures converge. We discuss the torus case in case 1 and the sphere case in case 2.
\begin{description}
\item[Case 1] Given a map $u:\mathcal{T}_{\tau}\to M$, with $\tau\in\mathcal{M}$, and a collection of maps $\{(v_0,\mathcal{T}_{\tau_0}),\{v_j\}_{j=1}^n\big\}$, where $v_0$ is a harmonic torus, $v_0:\mathcal{T}_{\tau_0}\to M$, $\tau_0\in\mathcal{M}$, and $\{v_j\}_{j=1}^n$ are harmonic spheres, $v_j:S^2\to M$, we say that 
\[d_B((u,T_\tau),\big\{(v_0,\mathcal{T}_{\tau_0}),\{v_j\}\big\})<\epsilon\]
if there exist $B_{r_j}(x_j)\subset \mathcal{T}_{\tau}$, $r_j<1$, $1\leq j\leq n$, $D_0\in S^1\times S^1$, $D_j\in\text{PSL}(2,\mathbb{C})$,
such that the following holds:
\begin{enumerate}
    \item for $i\neq j$, if $B_{r_i}(x_i)\cap B_{r_j}(x_j)\neq\emptyset$, we have either $B_{r_i}(x_i)\subset B_{r_j^2}(x_j)$ or $B_{r_j}(x_j)\subset B_{r_i^2}(x_i)$. Let
    \[\Omega_j=B_{r_j^2}(x_j)\setminus\Big\{\bigcup B_{r_i}(x_i),\: B_{r_i}(x_i)\subset B_{r_j^2}(x_j)\Big\}.\]

    $\{\Omega_j\}_{j=1}^n$ is a collection of pairwise disjoint domains in $T_\tau$.
    \item Since $\mathbb{C}$ is confromally equivalent to $S^2\setminus\{\text{north pole}\}$, we can regard $v_j$ as harmonic maps defined on $\mathbb{C}$ via stereographic projection.
    \begin{equation}\label{bubble}
  |\tau-\tau_0|<\epsilon,        
    \end{equation}
    \begin{equation}\label{bubble1}
        \begin{split}
            \Big(\int_{T_\tau\setminus\bigcup_{j}B_{r_j}(x_j)}|\nabla u-&\nabla (v_0\circ D_0\circ\iota_{\tau,\tau_0})|^2\Big)^{1/2}\\
            &+\Big(\int_{\bigcup_jB_{r_j}(x_j)}|\nabla (v_0\circ D_0\circ\iota_{\tau,\tau_0})|^2\Big)^{1/2}<\epsilon/3,
        \end{split}
    \end{equation}
    \begin{equation}\label{bubble2}
        \sum_{j=1}^n\Big(\int_{\Omega_j}|\nabla u-\nabla(v_j\circ D_j)|^2\Big)^{1/2}+\Big(\int_{\mathbb{C}\setminus\Omega_j}|\nabla(v_j\circ D_j)|^2\Big)^{1/2}<\epsilon/3,
    \end{equation}
    \begin{equation}\label{bubble3}
        \Big(\int_{\bigcup_j\{B_{r_j}(x_j)\setminus B_{r_j^2}(x_j)\}}|\nabla u|^2\Big)^{1/2}<\epsilon/3.
    \end{equation}
\end{enumerate}
\eqref{bubble} shows that the confromal structure of $T_\tau$ is "close" to $\mathcal{T}_{\tau_0}$, which is the domain of $v_0$. \eqref{bubble1} and \eqref{bubble2} are for describing the convergence of the body map and bubble map in $W^{1,2}$ sense on the restricted domain $T_\tau\setminus\bigcup_{j}B_{r_j}(x_j)$ and $\Omega_j$, while \eqref{bubble3} represents the energy loss on the neck domain $B_{r_j}(x_j)\setminus B_{r_j^2}(x_j)$. From Theorem \ref{torusconvergence} we know that if a sequence of maps $\{(u_i,\mathcal{T}_{\tau_i})\}$ converges to $\big\{(v_0,\mathcal{T}_{\tau_0}),\{v_j\}_{j=1}^n\big\}$ in a bubble tree sense, then \[d_B((u_i,\mathcal{T}_{\tau_i}),\big\{(v_0,\mathcal{T}_{\tau_0}),\{v_j\}_{j=1}^n\big\})\to 0,\quad i\to\infty.\]
\item[Case 2]
Given a map $u:T_\tau\to M$, with $\tau\in\mathcal{M}$, and a collection of harmonic spheres $\{v_j\}_{j=1}^n$, $v_j:S^2\to M$, we say that 
\[d_B((u,T_\tau),\{v_j\}_{j=1}^n)<\epsilon\]
if there exist $B_{r_j}(x_j)\subset \mathcal{T}_{\tau}$, $r_j<1$, $D_j\in\text{PSL}(2,\mathbb{C})$, for $1\leq j\leq n$,
such that the following holds:
\begin{enumerate}
    \item For $i\neq j$, if $B_{r_i}(x_i)\cap B_{r_j}(x_j)\neq\emptyset$, we have either $B_{r_i}(x_i)\subset B_{r_j^2}(x_j)$ or $B_{r_j}(x_j)\subset B_{r_i^2}(x_i)$. Let
    \[\Omega_j=B_{r_j^2}(x_j)\setminus\Big\{\bigcup B_{r_i}(x_i),\: B_{r_i}(x_i)\subset B_{r_j^2}(x_j)\Big\}.\]
    $\{\Omega_j\}_{j=1}^n$ is a collection of pairwise disjoint domains in $T_\tau$.
    \item 
     We argue as above and regard $v_j$ as harmonic maps defined on $\mathbb{C}$. Since we are considering the case that body map is degenerated, we use exponential map to conformally transform $\mathbb{C}\setminus\{0\}\simeq S^1\times\mathbb{R}$, and regard the body map $v_1$ as harmonic map defined on $S^1\times\mathbb{R}$. We don't strictly distinguish $v_1$ and $v_1\circ e^z$ here in abuse of notation. Also by changing parameters (see the proof of Theorem \ref{torusconvergence}) we can consider $T_\tau\simeq S^1\times[-L,L]\subset S^1\times\mathbb{R}.$
    \begin{equation}\label{bubble4}
        \begin{split}
            \Big(\int_{T_\tau\setminus\bigcup_{j}B_{r_j}(x_j)}|\nabla u-&\nabla (v_1\circ D_1)|^2\Big)^{1/2}\\
            &+\Big(\int_{S^1\times\mathbb{R}\setminus (T_\tau\setminus\bigcup_jB_{r_j}(x_j)))}|\nabla (v_1\circ D_1)|^2\Big)^{1/2}<\epsilon/3,
        \end{split}
    \end{equation}
    \begin{equation}\label{bubble5}
        \sum_{j=1}^n\Big(\int_{\Omega_j}|\nabla u-\nabla(v_j\circ D_j)|^2\Big)^{1/2}+\Big(\int_{\mathbb{C}\setminus\Omega_j}|\nabla(v_j\circ D_j)|^2\Big)^{1/2}<\epsilon/3,
    \end{equation}
    \begin{equation}\label{bubble6}
        \Big(\int_{\bigcup_j\{B_{r_j}(x_j)\setminus B_{r_j^2}(x_j)\}}|\nabla u|^2\Big)^{1/2}<\epsilon/3.
    \end{equation}
\end{enumerate}
Similar to Case 1, we see that \eqref{bubble4} and \eqref{bubble5} are for describing the convergence of the degenerated body map and bubble map in $W^{1,2}$ sense on the restricted domain $T_\tau\setminus\bigcup_{j}B_{r_j}(x_j)$ and $\Omega_j$, while \eqref{bubble6} represents the energy loss on the neck domain $B_{r_j}(x_j)\setminus B_{r_j^2}(x_j)$. From Theorem \ref{torusconvergence} we know that if a sequence of maps $\{(u_i,\mathcal{T}_{\tau_i})\}$ converges to $\{v_j\}_{j=1}^n$ in a bubble tree sense with degenerated body map, then
\[d_B((u_i,\mathcal{T}_{\tau_i}),\{v_j\}_{j=1}^n)\to 0,\quad i\to\infty.\]
\end{description}
\end{definition}
\begin{definition}[Equivalent Class]\label{equivalent class}
Given harmonic maps $u:\Sigma\to M$ and $v:\Sigma\to M$. We have $v\sim u$ if there exists conformal automorphism $g$ of $\Sigma$ such that $v\circ g=u$. We denote $[u]$ by the equvalent class of $u$.
\end{definition}
Although the bubble norm $d_B(\cdot,\cdot)$ is not a norm, but if a map $u:T_\tau\to M$ is close to a collection of harmonic maps in the bubble tree sense, then for any map $u':\mathcal{T}_{\tau'}\to M$ that's sufficiently close to $u$ in $W^{1,2}$ sense, $u'$ is also close to the same collection of harmonic maps in the bubble tree sense. 
\begin{claim}\label{open unstable}
Given $u:T_\tau\to M$, and a collection of harmonic maps $\{(v_0,\mathcal{T}_{\tau_0}),\{v_j\}_{j=1}^n\big\}$. If $d_B((u_i,\{(v_0,\mathcal{T}_{\tau_0}),\{v_j\}_{j=1}^n\big\})<\epsilon$ Then there exists $\delta(u)>0$, which depends on $(u,T_\tau)$ and $\{(v_0,\mathcal{T}_{\tau_0}),\{v_j\}_{j=1}^n\big\}$, such that for any map $u':\mathcal{T}_{\tau'}\to M$, if $|\tau-\tau'|<\delta(u)$ and
\[\int_{\mathcal{T}_{\tau'}}|\nabla u'-\nabla( u\circ\iota_{\tau,\tau'})|^2<\delta(u),\]
then 
\[d_B((u',\mathcal{T}_{\tau'}),\{(v_0,\mathcal{T}_{\tau_0}),\{v_j\}_{j=1}^n\big\})<\epsilon.\]
\end{claim}
\begin{proof}
$d_B((u_i,\{(v_0,\mathcal{T}_{\tau_0}),\{v_j\}_{j=1}^n\big\})<\epsilon$ implies that there exist $B_{r_j}(x_j)\subset T_\tau$, pairwise disjoint domains $\{\Omega_j\}$, and conformal automorphisms $\{D_j\}$ such that the inequalities \eqref{bubble1}, \eqref{bubble2}, and \eqref{bubble3} hold. We can assume $\tau'$ is sufficiently close to $\tau$ such that $B_{r_j}(x_j)\subset \mathcal{T}_{\tau'}$ and $E(u\circ\iota_{\tau',\tau},\mathcal{T}_{\tau'})-E(u,T_\tau)$ is sufficiently small. Then we have 
\begin{equation}
\begin{split}
&\Big(\int_{\mathcal{T}_{\tau'}\setminus\bigcup_{j}B_{r_j}(x_j)}|\nabla u'-\nabla (v_0\circ D_0\circ\iota_{\tau',\tau_0})|^2\Big)^{1/2}\leq\\
            &\Big(\int_{\mathcal{T}_{\tau'}}|\nabla u'-\nabla( u\circ\iota_{\tau,\tau'})|^2\Big)^{1/2}
            +\Big(\int_{T_\tau\setminus\bigcup_{j}B_{r_j}(x_j)}|\nabla u-\nabla (v_0\circ D_0\circ\iota_{\tau,\tau_0})|^2\Big)^{1/2}.
\end{split}
\end{equation}
Thus for $\delta(u)>0$ sufficiently small we have
    \begin{equation}
        \begin{split}
            \Big(\int_{\mathcal{T}_{\tau'}\setminus\bigcup_{j}B_{r_j}(x_j)}|\nabla u'-&\nabla (v_0\circ D_0\circ\iota_{\tau',\tau_0})|^2\Big)^{1/2}\\
            &+\Big(\int_{\bigcup_jB_{r_j}(x_j)}|\nabla (v_0\circ D_0\circ\iota_{\tau',\tau_0})|^2\Big)^{1/2}<\epsilon/3.
        \end{split}
    \end{equation}
Similarly, \eqref{bubble2} and \eqref{bubble3} hold for the same domain and conformal automorphisms with $u$ replaced by $u'$. Thus we have $d_B((u',\mathcal{T}_{\tau'}),\{(v_0,\mathcal{T}_{\tau_0}),\{v_j\}_{j=1}^n\big\})<\epsilon.$
\end{proof}
\section{Unstable Lemma}
The main focus of the section is Lemma $\ref{unstable}$ and $\ref{unstable1}$: proving the energy is concave for any map with some flat torus as its domain, that are sufficiently close to either a finite collection of one harmonic torus and possibly several harmonic spheres or a finite collection of harmonic spheres in bubble tree sense. This corresponds to the two possible convergence result given by Theorem \ref{torusconvergence}. Now we focus on the case where the conformal structure converges and the limit maps are one harmonic torus and possibly several harmonic spheres.
Consider a closed manifold $M$ of dimension at least three which is isometrically embedded in $\mathbb{R}^N$, given a collection of harmonic maps $\{(v_0,\mathcal{T}_{\tau_0}),\{v_j\}_{j=1}^n\}$, where $v_0$ is a harmonic torus, $v_0:\mathcal{T}_{\tau_0}\to M$, $\tau_0\in\mathcal{M}$, and $\{v_j\}_{j=1}^n$ are harmonic spheres. Assume that $\text{Index}(v_j)=k_j$ for each $j\in\{0,1,...,n\}$, by Definition \ref{indexdef}, there exists subspace  $\{X_{j,i}\}_{i=1}^{k_j}\subset\Gamma(v_j^{-1}TM)$ on which the index form is negative definite, by rescaling $\{X_{j,i}\}_{i=1}^{k_j}$ (which we still denote by $\{X_{j,i}\}_{i=1}^{k_j}$), we have that for 
\[E_{v_0}(s_0,\mathcal{T}_{\tau_0})=E(\Pi\circ\big(v_0+\sum_{i=1}^{k_0}s_{0,i}X_{0,i}\big),\mathcal{T}_{\tau_0}),\quad s_0=(s_{0,1},...,s_{0,k_0})\in\bar{B}^{k_0},\]
and 
\[E_{v_j}(s_j,\mathbb{C})=E(\Pi\circ\big(v_j+\sum_{i=1}^{k_j}s_{j,i}X_{j,i}\big),\mathbb{C}),\quad s_j=(s_{j,1},...,s_{j,k_j})\in\bar{B}^{k_j},\:j=1,...,n,\]
there exist $c_0<1$ such that 
\begin{equation}\label{unstableassumption}
-\frac{1}{2c_0}\text{Id}< D^2_{s_0}E_{v_0}(s_0,\mathcal{T}_{\tau_0})< -2c_0\text{Id},\quad\forall s_0\in\bar{B}^{k_0},
\end{equation}
and
\begin{equation}\label{unstableassumption1}
-\frac{1}{2c_0}\text{Id}< D^2_{s_j}E_{v_j}(s_j,\mathbb{C})< -2c_0\text{Id},\quad\forall s_j\in\bar{B}^{k_j},
\end{equation}
for each $j=1,...,n$.
\begin{claim}\label{forunstable}
Let $\{(v_0,\mathcal{T}_{\tau_0}),\{v_j\}_{j=1}^n\}$ and $\{X_{j,i}\}$ be given as above.
Given $\epsilon>0$, there exists $\delta(\epsilon)>0$, such that for any $B_{\rho}(x)\subset\mathbb{C}$, if 
\begin{equation}
E(v_j,B_{\rho}(x))<\delta(\epsilon),    
\end{equation}
then we have 
\begin{equation}
    -\epsilon\text{Id}<D^2_{s_j}E_{v_j}(s_j,B_\rho(x))<\epsilon\text{Id},\text{ and }\int_{B_\rho(x)}|\nabla X_{j,i}|^2<\epsilon,
\end{equation}
for $i=1,...,k_j$, and $j=0,...,n.$
\end{claim}
\begin{proof}
Let
\[\Tilde{\rho}:=\inf\big\{\rho>0,\:\int_{B_\rho(x)}|\nabla X_{j,i}|^2>\frac{\epsilon}{2},B_{\rho}(x)\subset\mathbb{C}\big\}.\]
We know that $\Tilde{\rho}>0$. Let
\[\delta_1(\epsilon):=\inf\big\{E(v_j,B_{\Tilde{\rho}}(x)),\:x\in\mathbb{C}\big\}.\]
It's easy to check that if $E(v_j,B_{\rho}(x))<\delta_1(\epsilon)$ for some $B_\rho(x)\subset\mathbb{C}$, then it implies that $\rho\leq\Tilde{\rho}$, thus we have $\int_{B_\rho(x)}|\nabla X_{j,i}|^2<\epsilon$. Similarly, we can choose $\delta_2(\epsilon)>0$ such that if $E(v_j,B_{\rho}(x))<\delta_2(\epsilon)$ for some $B_\rho(x)\subset\mathbb{C}$, then it implies \[-\epsilon\text{Id}<D^2_{s_j}E_{v_j}(s_j,B_\rho(x))<\epsilon\text{Id}.\]
We finish the proof by choosing $\delta(\epsilon)=\min\{\delta_1(\epsilon),\delta_2(\epsilon)\}.$
\end{proof}
\begin{lemma}[Unstable Lemma for torus]\label{unstable} Let $\{(v_0,\mathcal{T}_{\tau_0}),\{v_j\}_{j=1}^n\}$, $\{X_{j,i}\}$ be given as above. There exists $\epsilon>0$, such that if a map $u:\mathcal{T}_{\tau}\to M$, $\tau\in\mathcal{M}$, $u\in W^{1,2}(T_\tau,M)$ has $d_B((u,\mathcal{T}_{\tau}),\{(v_0,\mathcal{T}_{\tau_0}),\{v_j\}_{j=1}^n\})<\epsilon.$
Then we can construct $\{\Tilde{X}_{j,i}\}\subset C^\infty(T_\tau,\mathbb{R}^N)$, so that for the smooth function \[E_u(\cdot,T_\tau):\bar{B}^k\to[0,\infty),\]
which is defined by
\[E_u(s,T_\tau)=E\big(\Pi\circ\big(u+\sum_{i=1,j=0}^{k_j,n}s_{j,i}\Tilde{X}_{j,i}\big),T_\tau\big),\quad s=(s_{0,1},...,s_{j,k_j},s_{j+1,1},...,s_{n,k_n})\in\bar{B}^k,\]
where $k=\sum_{j=0}^nk_j$, the following holds:
\begin{equation}\label{33}
    -\frac{1}{c_0}\text{Id}\leq D^2_sE_u(s,T_\tau)\leq -c_0\text{Id},\quad\forall s\in\bar{B}^k.
\end{equation}
Moreover, $E_u(\cdot,T_\tau)$ has a unique maximum at $m(u)\in B^k_{c_0/\sqrt{10}}$ and
\begin{equation}\label{34}
E_u(m(u),T_\tau)-\frac{1}{2c_0}|s-m(u)|^2\leq E_u(s,T_\tau)\leq E_u(m(u),T_\tau)-\frac{c_0}{2}|s-m(u)|^2,    
\end{equation}
for all $s\in\bar{B}^k.$
\end{lemma}
\begin{proof}
We assume that $d_B((u,T_\tau),\{(v_0,\mathcal{T}_{\tau_0}),\{v_j\}_{j=1}^n\})<\epsilon$ for some constant $\epsilon>0$ which will be chosen later, by Definition \ref{bubblenorm} we have the following
    \begin{equation}\label{1}
  |\tau-\tau_0|<\epsilon,        
    \end{equation}
    \begin{equation}\label{2}
        \begin{split}
            \Big(\int_{T_\tau\setminus\bigcup_{j}B_{r_j}(x_j)}|\nabla u-&\nabla (v_0\circ D_0\circ\iota_{\tau,\tau_0})|^2\Big)^{1/2}\\
            &+\Big(\int_{\bigcup_jB_{r_j}(x_j)}|\nabla (v_0\circ D_0\circ\iota_{\tau,\tau_0})|^2\Big)^{1/2}<\epsilon/3,
        \end{split}
    \end{equation}
    \begin{equation}\label{3}
        \sum_{j=1}^n\Big(\int_{\Omega_j}|\nabla u-\nabla(v_j\circ D_j)|^2\Big)^{1/2}+\Big(\int_{\mathbb{C}\setminus\Omega_j}|\nabla(v_j\circ D_j)|^2\Big)^{1/2}<\epsilon/3,
    \end{equation}
    \begin{equation}\label{4}
        \Big(\int_{\bigcup_j\{B_{r_j}(x_j)\setminus B_{r_j^2}(x_j)\}}|\nabla u|^2\Big)^{1/2}<\epsilon/3.
    \end{equation}
Now we define the following piecewise smooth cutoff function, which was introduced by Choi and Schoen \cite{CHS},  $\eta_j:[0,\infty)\to[0,1]$:
		\begin{equation}\label{388}
		\eta_j(r) =\left\{ \begin{array}{rcl}
		0, & \mbox{for}	& r<r_j^2, \\
		2-(\log r)/(\log r_j),   & \mbox{for} & r_j^2\leq r \leq r_j,\\
		1, & \mbox{for} & r>r_j, 
		\end{array}\right.
		\end{equation}
		so that
		\[\frac{d\eta_j}{dr}(r) =\left\{ \begin{array}{rcl}
		0, & \mbox{for}	& r<r_j^2, \\
		-1/r(\log r_j),   & \mbox{for} & r_j^2\leq r \leq r_j,\\
		0, & \mbox{for} & r>r_j, 
		\end{array}\right.\]
		
		and 
		\begin{equation}\label{cutoff}
		\int_{0}^{2\pi}\int_{r_j^2}^{r_j}\Big(\frac{d\eta_j}{dr}(r)\Big)^2rdrd\theta=-\frac{2\pi}{\log r_j}.	
		\end{equation}
We can choose a local conformal coordinate $z=x+\sqrt{-1}y$ centered at $x_j\in T_\tau$ and let $f(z)=\sqrt{x^2+y^2}$ so that we can define the $r_j$-ball $B_{r_j}(x_j)$ about $x_j$ as the set of points satisfying $r\leq r_j$. The radial cutoff function we will use in our constructions is then the map 
\begin{equation}\label{38}
    \zeta_j:B_{r_j}(x_j)\to[0,1],\quad\text{defined by }\zeta_j=\eta_j\circ f.
\end{equation}
For distinct $j$ we have the local coordinate centered at distinct $x_j$, thus the function $r$ in \eqref{38} is different for each $j$, but we don't strictly distinguish them and denote them by $r$ for all $j$ in abuse of notation and terminology. 
\begin{claim}\label{unstabletorus}
There exists $\epsilon>0$, such that for $d_B((u,T_\tau),\{(v_0,\mathcal{T}_{\tau_0}),\{v_j\}_{j=1}^n\})<\epsilon$, and
\begin{equation}\label{newvf}
    \Tilde{X}_{0,i}:=\prod_{j=1,...,n}\zeta_j X_{0,i}\circ D_0\circ\iota_{\tau,\tau_0},
\end{equation}
where $\zeta_j$ is given by \eqref{38}, we have 
\begin{equation}\label{40}
    -c_0\text{Id}< D^2_{s_0}E(\Pi\circ(u+\sum_{i=1}^{k_0}s_{0,i}\Tilde{X}_{0,i}),T_\tau)< -\frac{1}{c_0}\text{Id},\quad\forall s_0=(s_{0,1},...,s_{0,k_0})\in\bar{B}^{k_0}.
\end{equation}
\end{claim}
\begin{proof}[proof of the claim]
To simplify the notation, let
 \begin{align}
  &\delta^2_{s,t}E_{u}(s\Tilde{X}_{0,l_1}+t\Tilde{X}_{0,l_2},T_\tau):=\frac{d^2}{dsdt}E(\Pi\circ(u+s\Tilde{X}_{0,l_1}+t\Tilde{X}_{0,l_2}),T_\tau),\\
  &\delta^2_{s}E_{u}(s\Tilde{X}_{0,l_1},T_\tau):=\frac{d^2}{ds^2}E(\Pi\circ(u+s\Tilde{X}_{0,l_1}),T_\tau),
 \end{align}
for $\{l_1,l_2\}\subset\{1,...,k_0\}$. We observe that to prove \eqref{40} is to show that we can choose $\epsilon>0$ to make the following term
$\delta^2_{s,t}E_{u}(s\Tilde{X}_{0,l_1}+t\Tilde{X}_{0,l_2},T_\tau)-\delta^2_{s,t}E_{v_0}(sX_{0,l_1}+tX_{0,l_2},\mathcal{T}_{\tau_0})$
sufficiently small for all $s,t\in[0,1]$, $\{l_1,l_2\}\subset\{1,...,k_0\}$. It's enough to show the following term
\[\delta^2_{s}E_{u}(s\Tilde{X}_{0,i},T_\tau)-\delta^2_{s}E_{v_0}(sX_{0,i},\mathcal{T}_{\tau_0}),\quad i=1,...,k_0,\] 
could be made sufficiently small by choosing $\epsilon>0$. We observe that 
\begin{equation}
    \begin{split}
        \delta^2_{s}E_{u}(s\Tilde{X}_{0,i},T_\tau)&=\delta^2_{s}E_{u}(s\Tilde{X}_{0,i},T_\tau\setminus\cup_{i=1}^nB_{r_j}(x_j))\\
        &+\delta^2_{s}E_{u}(s\Tilde{X}_{0,i},\cup_{i=1}^nB_{r_j}(x_j)\setminus B_{r_j^2}(x_j)).
    \end{split}
\end{equation}
Which implies 
\begin{equation}\label{44}
\begin{split}
    |\delta^2_{s}&E_{u}(s\Tilde{X}_{0,i},T_\tau)-\delta^2_sE_{v_0}(sX_{0,i},\mathcal{T}_{\tau_0})|\\
    &<|\delta^2_{s}E_{u}(s\Tilde{X}_{0,i},T_\tau\setminus\cup_{i=1}^nB_{r_j}(x_j))-\delta^2_{s}E_{\tilde{v}_0}(sX_{0,i}\circ D_0\circ\iota_{\tau,\tau_0},T_\tau\setminus\cup_{i=1}^nB_{r_j}(x_j))|\\
    &+|\delta^2_{s}E_{\tilde{v}_0}(sX_{0,i}\circ D_0\circ\iota_{\tau,\tau_0},\cup_{i=1}^nB_{r_j}(x_j))|\\
    &+|\delta^2_{s}E_{\tilde{v}_0}(sX_{0,i}\circ D_0\circ\iota_{\tau,\tau_0},\mathcal{T}_{\tau})-\delta^2_sE_{v_0}(sX_{0,i},\mathcal{T}_{\tau_0})|\\
    &+|\delta^2_{s}E_{u}(s\Tilde{X}_{0,i},\cup_{i=1}^nB_{r_j}(x_j)\setminus B_{r_j^2}(x_j))|,
\end{split}    
\end{equation}
where $\Tilde{v}_0:=v_0\circ D_0\circ\iota_{\tau,\tau_0}$. We can make $|\delta^2_{s}E_{u}(s\Tilde{X}_{0,i},T_\tau)-\delta^2_sE_{v_0}(sX_{0,i},\mathcal{T}_{\tau_0})|$ sufficiently small by making each term in \eqref{44} sufficiently small. We will discuss how to bound each term in \eqref{44} in the following.

\begin{description}
\item[1] $|\delta^2_{s}E_{u}(s\Tilde{X}_{0,i},T_\tau\setminus\cup_{i=1}^nB_{r_j}(x_j))-\delta^2_{s}E_{\tilde{v}_0}(sX_{0,i}\circ D_0\circ\iota_{\tau,\tau_0},T_\tau\setminus\cup_{i=1}^nB_{r_j}(x_j))|$

We recall from \eqref{energyTaylor2} for any $X\in C^\infty(T_\tau,\mathbb{R}^N)$ we have the following,  
\begin{equation}\label{45}
    \begin{split}
        \frac{d^2}{ds^2}E(\Pi\circ(u+sX),T_\tau)=&2\Big\{\int_{T_\tau}\langle d\Pi_u(\nabla X),d\Pi_u(\nabla X)\rangle \\
	&+\int_{T_\tau}2\langle \text{Hess}\Pi_u(X,\nabla u),d\Pi_u(\nabla X)\rangle \\
	&+\int_{T_\tau}\langle \text{Hess}\Pi_u(X,\nabla u),\text{Hess}\Pi_u(X,\nabla u)\rangle \\
	&+\int_{T_\tau}\langle\nabla u,\big( 2\text{Hess}\Pi_u(X,\nabla X)+\nabla^3\Pi_u(X,X,\nabla u)\big)\rangle \Big\}\\
	&+o(s).
    \end{split}
\end{equation}
Since $X_{0,1}\in\Gamma(v_0^{-1}TM)$, so $d\Pi_{v_0}(\nabla X_{0,1})=\nabla X_{0,1}$, which implies 
\begin{equation}\label{47}
d\Pi_{\Tilde{v}_0}\big(\nabla(X_{0,1}\circ D_0\circ\iota_{\tau,\tau_0})\big)=\nabla(X_{0,1}\circ D_0\circ\iota_{\tau,\tau_0}).    
\end{equation}
Moreover, $\zeta_j$ has value $1$ on the domain $T_\tau\setminus\cup_{i=1}^nB_{r_j}(x_j)$, so $\Tilde{X}_{0,1}=X_{0,1}\circ D_0\circ\iota_{\tau,\tau_0}$ on $T_\tau\setminus\cup_{i=1}^nB_{r_j}(x_j)$, thus we have
\begin{equation}\label{48}
\int_{T_\tau\setminus\cup_{i=1}^nB_{r_j}(x_j)}|\nabla\Tilde{X}_{0,1}|^2=\int_{T_\tau\setminus\cup_{i=1}^nB_{r_j}(x_j)}|\nabla\big(X_{0,1}\circ D_0\circ\iota_{\tau,\tau_0}\big)|^2. 
\end{equation}
We observe that the only term in \eqref{45} not bounded by the energy difference of $u$ and $\Tilde{v}_0$ on the domain $T_\tau\setminus\cup_{i=1}^nB_{r_j}(x_j)$ is the term $\int|d\Pi_u(\nabla X)|^2$. Since $M$ is a smooth manifold and $d\Pi_u(\nabla X)$ is the orthogonal projection of $\nabla X$ on $T_uM$, by \eqref{2} we know that $u$ is close to $\tilde{v}_0$ on $T_\tau\setminus\cup_{i=1}^nB_{r_j}(x_j)$ in $W^{1,2}$ sense, thus implies that the term $\int|d\Pi_u(\nabla X)|^2$ can be bounded by $\int|d\Pi_{\tilde{v}_0}(\nabla X)|^2$ and \eqref{2} for any $X\in C^\infty(T_\tau,\mathbb{R}^N)$. Then $\eqref{47}$ and $\eqref{48}$ implies that 
\begin{equation}
\begin{split}
|\delta^2_{s}E_{u}(s\Tilde{X}_{0,i},T_\tau\setminus\cup_{i=1}^nB_{r_j}(x_j))-\delta^2_{s}E_{\tilde{v}_0}(sX_{0,i}&\circ D_0\circ\iota_{\tau,\tau_0},T_\tau\setminus\cup_{i=1}^nB_{r_j}(x_j))|\\
&<C(d_B(u,\{(v_0,\mathcal{T}_{\tau_0}),\{v_j\}_{j=1}^n\}))     
\end{split}
\end{equation}
for a continuous nonnegative function $C$ depending on the manifold $M$ and $\{X_{0,i}\}_{i=1}^n$, such that $C(0)=0$. 
\item[2]$|\delta^2_{s}E_{\tilde{v}_0}(sX_{0,i}\circ D_0\circ\iota_{\tau,\tau_0},\mathcal{T}_{\tau})-\delta^2_sE_{v_0}(sX_{0,i},\mathcal{T}_{\tau_0})|$

Since $\iota_{\tau,\tau_0}$ is the linear map of $\mathbb{C}$ keeping $1$ and sending $\tau$ to $\tau_0$, Conformal invariance of energy imply that
\begin{equation}
    |\delta^2_{s}E_{\tilde{v}_0}(sX_{0,i}\circ D_0\circ\iota_{\tau,\tau_0},\mathcal{T}_{\tau})-\delta^2_sE_{v_0}(sX_{0,i},\mathcal{T}_{\tau_0})|\leq C_1(|\tau-\tau_0|),
\end{equation}
where $C_1$ is a nonnegative continuous function, depending on $M$ and $\{X_{0,i}\}_{i=1}^n$, such that $C_1(0)=0$. 
\item[3]
$|\delta^2_{s}E_{\tilde{v}_0}(sX_{0,i}\circ D_0\circ\iota_{\tau,\tau_0},\cup_{i=1}^nB_{r_j}(x_j))|$

By Claim \ref{forunstable}, there exists $\delta(c_0/10)>0$, such that  
\[-\frac{c_0}{10}\text{Id}<D^2_{s_j}E_{v_j}(s_j,B_\rho(x))<\frac{c_0}{10}\text{Id},\]
for $E(v_j,B_{\rho}(x))<\delta(\frac{c_0}{10})$. Since energy is invariant under conformal diffeomorphism of the domain, the $\delta(\epsilon)$ chosen in Claim \ref{forunstable} remains the same for $v_j\circ D_j$, where $D_0\in S^1\times S^1$ and $D_j\in\text{PSL}(2,\mathbb{C})$ for each $j=1,...,n$. If the bubble norm $d_B(u,\{(v_0,\mathcal{T}_{\tau_0}),\{v_j\}_{j=1}^n\})<(\delta(\epsilon))^{1/2}$, then \eqref{2} give us $E(v_0\circ D_0\circ\iota_{\tau,\tau_0},\cup_{i=1}^nB_{r_j}(x_j)))<\delta(\epsilon)$, so Claim \ref{forunstable} implies the following inequality: 
\begin{equation}
    |\delta^2_{s}E_{\tilde{v}_0}(sX_{0,i}\circ D_0\circ\iota_{\tau,\tau_0},\cup_{i=1}^nB_{r_j}(x_j))|<\frac{c_0}{10}.
\end{equation}
\item[4]
$|\delta^2_{s}E_{u}(s\Tilde{X}_{0,i},\cup_{i=1}^nB_{r_j}(x_j)\setminus B_{r_j^2}(x_j))|$ 

By \eqref{45} we know that the only term not bounded by the energy of $u$ on the domain $\cup_{i=1}^nB_{r_j}(x_j)\setminus B_{r_j^2}(x_j))$ is the term $\int|d\Pi_u(\nabla X)|^2$. Since $d\Pi_u(\nabla X)$ is the orthogonal projection of $\nabla X$ on $T_uM$, thus it's bounded by $\int|\nabla X|^2$ on $\cup_{i=1}^nB_{r_j}(x_j)\setminus B_{r_j^2}(x_j).$
\begin{equation}\label{50}
\begin{split}
|\delta^2_{s}E_{u}&(s\Tilde{X}_{0,i},\cup_{i=1}^nB_{r_j}(x_j)\setminus B_{r_j^2}(x_j))|\\
<&C_2\big(E(u,\cup_{i=1}^nB_{r_j}(x_j)\setminus B_{r_j^2}(x_j))\big)+\int_{\cup_{i=1}^nB_{r_j}(x_j)\setminus B_{r_j^2}(x_j)}|\nabla\Tilde{X}_{0,i}|^2\\
=&C_2\big(E(u,\cup_{i=1}^nB_{r_j}(x_j)\setminus B_{r_j^2}(x_j))\big)\\
+&\int_{\cup_{i=1}^nB_{r_j}(x_j)\setminus B_{r_j^2}(x_j)}|\nabla\Big(\prod_{j=1,...,n}\zeta_j X_{0,i}\circ D_0\circ\iota_{\tau,\tau_0}\Big)|^2,
\end{split}
\end{equation}
where $C_2$ is a nonnegative continuous function depending on $M$ and $\{X_{0,i}\}_{i=1}^n$ such that $C_2(0)=0$. By how we construct $\zeta_j$ \eqref{38} and \eqref{cutoff}, we have the following inequality
\begin{equation}
    \begin{split}
        \int_{\cup_{i=1}^nB_{r_j}(x_j)\setminus B_{r_j^2}(x_j)}&|\nabla\Big(\prod_{j=1,...,n}\zeta_j X_{0,i}\circ D_0\circ\iota_{\tau,\tau_0}\Big)|^2\\
        &=\bigcup_{j=1}^n\int_{B_{r_j}(x_j)\setminus B_{r_j^2}(x_j)}|\nabla(\zeta_jX_{0,i}\circ D_0\circ\iota_{\tau,\tau_0})|^2\\
        &<\bigcup_{j=1}^n\Big\{C_3\int_{B_{r_j}(x_j)\setminus B_{r_j^2}(x_j)}|\nabla\zeta_j|^2 \\
        &+2\int_{B_{r_j}(x_j)\setminus B_{r_j^2}(x_j)}|\nabla(X_{0,i}\circ D_0\circ\iota_{\tau,\tau_0})|^2\Big\},
    \end{split}
\end{equation}
where $C_3$ is a constant depending on $X_{0,i}$. Choosing $r$  such that $-\frac{2\pi}{\log r}<\frac{c_0}{10nC_3}$. 
\begin{equation}\label{52}
\epsilon_2:=\min\{E(v_0\circ D_0,B_{\rho}(x)),B_{\rho}(x)\subset \mathcal{T}_{\tau_0},\rho\geq r/2,D_0\in S^1\times S^1\}.
\end{equation}
We know that $\epsilon_2>0$. Then $d_B(u,\{(v_0,\mathcal{T}_{\tau_0}),\{v_j\}_{j=1}^n\})<\min\{\epsilon_2^{1/2},\delta(c_0/20n)\}$ implies that
\[\int_{\bigcup_jB_{r_j}(x_j)}|\nabla (v_0\circ D_0\circ\iota_{\tau,\tau_0})|^2<\min\{\epsilon_2^{1/2},\delta(c_0/20)\},\]
by \eqref{52} and Claim \ref{forunstable} we have the following inequality
\begin{equation}\label{53}
    \begin{split}
        \bigcup_{j=1}^n\Big\{C_3\int_{B_{r_j}(x_j)\setminus B_{r_j^2}(x_j)}|\nabla\zeta_j|^2+2\int_{B_{r_j}(x_j)\setminus B_{r_j^2}(x_j)}|&\nabla(X_{0,i}\circ D_0\circ\iota_{\tau,\tau_0})|^2\Big\}\\
        &<n\big(C_3\times\frac{c_0}{10nC_3}+\frac{c_0}{20n}\big)\\
        &<\frac{c_0}{5}.
    \end{split}
\end{equation}
\end{description}
By combining all the above estimates we know that if the bubble norm is bounded by $\epsilon>0$, i.e., $d_B((u,\mathcal{T}_{\tau}),\{(v_0,\mathcal{T}_{\tau_0}),\{v_j\}_{j=1}^n\})<\epsilon$, then we can choose $\epsilon$ sufficiently small so $\delta^2_{s,t}E_{u}(s\Tilde{X}_{0,l_1}+t\Tilde{X}_{0,l_2},T_\tau)-\delta^2_{s,t}E_{v_0}(sX_{0,l_1}+tX_{0,l_2},\mathcal{T}_{\tau_0})$ is sufficiently small for all ${l_1,l_2}\subset\{1,...,n\}$ thus implies 
\[-c_0\text{Id}< D^2_{s_0}E(\Pi\circ(u+\sum_{i=0}^{k_0}s_{0,i}\Tilde{X}_{0,i}),T_\tau)< -\frac{1}{c_0}\text{Id},\quad\forall s_0=(s_{0,1},...,s_{0,k_0})\in\bar{B}^{k_0},\]
which is the desired result.
\end{proof}
Claim \ref{unstabletorus} implies that for $d_B((u,\mathcal{T}_{\tau}),\{(v_0,\mathcal{T}_{\tau_0}),\{v_j\}_{j=1}^n\})<\epsilon$, if the body map $v_0:\mathcal{T}_{\tau_0}\to M$ has Index$(v_0)=k_0$, then we can construct $\Tilde{X}_{0,i}$, $i=1,...,k_0$, so that $d\Pi_u(\Tilde{X}_{0,i})$ form a $k_0$ dimensional subspace of $\Gamma(u^{-1}TM)$ on which the energy of $u$ is concave. For $v_j:S^2\to M$ Index$(v_j)=k_j$, $j=1,...,n$, we can construct $\Tilde{X}_{j,i}$ similarly as we construct $\Tilde{X}_{0,i}$. The only difference is the cutoff function $\zeta_j$. In Claim \ref{unstabletorus}, we argue that $d_B((u,\mathcal{T}_{\tau}),\{(v_0,\mathcal{T}_{\tau_0}),\{v_j\}_{j=1}^n\})<\epsilon$ can imply the following term 
\[|\delta^2_{s}E_{u}(s\Tilde{X}_{0,i},\cup_{i=1}^nB_{r_j}(x_j)\setminus B_{r_j^2}(x_j))|\]
to be sufficiently small. To do so, we need to bound the term in \eqref{50}. If we consider the harmonic sphere $v_j:S^2\to M$ instead of harmonic torus $v_0$, with Index($v_j$)=$k_j$ and the corresponding $\{X_{i,j}\}_{i=1}^{k_j}\subset\Gamma(v_j^{-1}TM)$. We know that from Claim \ref{unstabletorus} that \eqref{50} can bounded by choosing the energy of $u$ on the domain $\cup_{j=1}^n(B_{r_j}(x_j)\setminus B_{r_j^2}(x_j))$ sufficiently small, and the energy of the cutoff function which has support on $\cup_{j=1}^n(B_{r_j}(x_j)\setminus B_{r_j^2}(x_j))$. If we choose some constant $r<1$ so that $-\frac{2\pi}{\log r}<\varepsilon$ for some given $\varepsilon>0$ and let 
\[\epsilon_2:=\min\{E(v_0\circ D_0,B_{\rho}(x)),B_{\rho}(x)\subset \mathcal{T}_{\tau_0},\rho\geq r/2,D_0\in S^1\times S^1\},\]
since $\epsilon_2>0$, then we have the energy of the cutoff function is bounded by $\varepsilon$ whenever the energy of the harmonic torus is bounded by $\epsilon_2$ on the domain which the cutoff function is defined on. But this estimate no longer works for harmonic sphere, since the conformal automorphism PSL(2,$\mathbb{C}$) of $\mathbb{C}\cup\{\infty\}\simeq S^2$ is not compact. However, we can overcome this by the conformal invariance property of energy functional. Now we fix $j\in\{1,...,n\}$. We choose $r$ so that $-\frac{2\pi}{\log r}<\varepsilon$ for some given $\varepsilon>0$. Let $\Pi_S$ denote the stereographic projection from $S^2\setminus\{\text{north pole}\}$ to $\mathbb{C}$, and let $\tilde{r}$ denote the radius of the geodesic ball $\Pi^{-1}_S(B_r(0))\subset S^2$, now we consider the following
\begin{equation}\label{e_3}
\epsilon_3:=\min\{E(v_j,B_{\rho}(x)),B_{\rho}(x)\subset S^2,\rho\geq\tilde{r}/2\},    
\end{equation}
where $B_\rho(x)$ is the geodesic ball centered at $x$ with radius $\rho$, we don't distinguish it from the ball in $\mathbb{C}$ in abuse of notation. Then $d_B((u,T_\tau),\{(v_0,\mathcal{T}_{\tau_0}),\{v_j\}_{j=1}^n\})<\epsilon_3^{1/2}$ implies that there exists $\Omega_j\subset\mathbb{C}$ and $D_j\in\text{PSL}(2,\mathbb{C})$ (see Definition \ref{bubblenorm}) such that 
\begin{equation}\label{56}
\Big(\int_{\Omega_j}|\nabla u-\nabla(v_j\circ D_j)|^2\Big)^{1/2}+\Big(\int_{\mathbb{C}\setminus\Omega_j}|\nabla(v_j\circ D_j)|^2\Big)^{1/2}<\frac{\epsilon_3^{1/2}}{3}.
\end{equation}
To make the discussion easier, we assume that there's only one $B_{r_l}(x_l)\subset B_{r_j^2}(x_j)$,
\begin{equation}
    \begin{split}
\Omega_j&=B_{r_j^2}(x_j)\setminus\Big\{\bigcup B_{r_i}(x_i),\: B_{r_i}(x_i)\subset B_{r_j^2}(x_j)\Big\}\\
&=B_{r_j^2}(x_j)\setminus B_{r_l}(x_l).
    \end{split}
\end{equation}
We can think of this case as the bubble map $v_j$ itself has another bubble map $v_l$ (see case 1 in Theorem \ref{torusconvergence}). In this case, we need to construct two cutoff function, one for the bubble map $v_j$ and another one for \textit{excluding} the bubble map $v_l$ of $v_j$. We focus on constructing the later first. So \eqref{56} implies that 
\[\int_{\mathbb{C}\setminus B_{r^2_j}(x_j)}|\nabla(v_j\circ D_j)|^2+\int_{ B_{r_l}(x_l)}|\nabla(v_j\circ D_j)|^2<\epsilon_3,\]
and by the conformal invariance of energy we have
\[\int_{\mathbb{C}\setminus D_j\circ B_{r^2_j}(x_j)}|\nabla v_j|^2+\int_{ D_j\circ B_{r_l}(x_l)}|\nabla v_j|^2<\epsilon_3.\]
By how $\epsilon_3$ is chosen, we know that $D_j\circ B_{r_l}(x_l)$ is a ball with radius $r_l'$ no larger than $r$, say $B'_{r'_l}(x_l)=D_j\circ B_{r_l}(x_l)$. We can construct the cutoff function $\zeta_{j,l}=\eta_{j,l}\circ f$ similarly to \eqref{388}, but with value $1$ on $\mathbb{C}\setminus D_j\circ B_{r_l}(x_l)=\mathbb{C}\setminus B_{r_l'}(x_l)$ instead of $\mathbb{C}\setminus B_{r_l}(x_l)$, and support on $(\mathbb{C}\setminus B_{r_l'^2}(x_l))\subset(\mathbb{C}\setminus D_j\circ B_{r_l^2}(x_l))$ i.e.,
\begin{equation}\label{59}
\eta_{j,l}(r) =\left\{ \begin{array}{rcl}
0, & \mbox{for}	& r<r_l'^2, \\
2-(\log r)/(\log r_l'),   & \mbox{for} & r_l'^2\leq r \leq r_l',\\
		1, & \mbox{for} & r>r_l', 
		\end{array}\right.
		\end{equation}
Thus we have
\begin{equation}\label{61}
\int_{D_j\circ B_{r_l}(x_l)}|\nabla \zeta_{j,l}|^2=\int_{B_{r_l}(x_l)}|\nabla(\zeta_{j,l}\circ D_j)|^2<\varepsilon.    
\end{equation}
So the inequality of \eqref{53} still holds for harmonic spheres. Now we construct the cutoff function for $v_j$. Since \eqref{56} implies that 
\[\int_{\mathbb{C}\setminus B_{r^2_j}(x_j)}|\nabla(v_j\circ D_j)|^2=\int_{\mathbb{C}\setminus D_j\circ B_{r^2_j}(x_j)}|\nabla v_j|^2<\epsilon_3.\]
By how we choose $\epsilon_3$, see \eqref{e_3}, we know that $S^2\setminus\Pi_S^{-1}(D_j\circ B_{r^2_j}(0))$ is a geodesic ball centered at north pole and has radius smaller than $\tilde{r}$. Since $\title{r}$ is the radius of $\Pi^{-1}_S(B_r(0))$ for some $r$ satisfying the inequality $-\frac{2\pi}{\log r}<\varepsilon$. Then we can construct the cutoff function $\zeta_{j,0}$ the same way as we constructed $\zeta_{j,l}$, such that $\zeta_{j,0}$ has value $1$ on $D_j\circ B_{r_j^2}(x_j)$, with support on $D_j\circ B_{r_j}(x_j)$, and its energy bounded by $\varepsilon$, i.e.,
\begin{equation}\label{622}
\int_{D_j\circ B_{r_j}(x_l)}|\nabla \zeta_{j,0}|^2=\int_{B_{r_j}(x_j)}|\nabla(\zeta_{j,0}\circ D_j)|^2<\varepsilon.    
\end{equation}
Now we consider the following
\begin{equation}
\Tilde{X}_{j,i}:=\Big(\zeta_{j,0}\zeta_{j,l} X_{j,i}\Big)\circ D_j, 
\quad i\in\{1,...,k_j\}.\end{equation}
Since $\zeta_{j,0}\circ D_j$ and $\zeta_{j,l}\circ D_j$ both has support on $T_\tau$, so $\Tilde{X}_{j,i}\in C^\infty(T_\tau,\mathbb{R}^N)$. 
We can argue similarly as Claim \ref{unstabletorus} and conclude that 
\begin{equation}\label{62}
-c_0\text{Id}< D^2_{s_j}E(\Pi\circ(u+\sum_{i=1}^{k_j}s_{j,i}\Tilde{X}_{j,i}),T_\tau)< -\frac{1}{c_0}\text{Id},\quad\forall s_j=(s_{j,1},...,s_{j,k_j})\in\bar{B}^{k_j}.    
\end{equation}
Combining \eqref{40} and \eqref{62} implies that 
\begin{equation}
    -\frac{1}{c_0}\text{Id}\leq D^2_sE_u(s,T_\tau)\leq -c_0\text{Id},\quad\forall s\in\bar{B}^k,
\end{equation}
where $k=\sum_{j=0}^nk_j$, and
\[E_u(s,T_\tau)=E\big(\Pi\circ\big(u+\sum_{i=1,j=1}^{k_j,n}s_{j,i}\Tilde{X}_{j,i}\big),T_\tau\big),\quad s=(s_{0,1},...,s_{j,k_j},s_{j+1,1},...,s_{n,k_n})\in\bar{B}^k,\]
which is the desired result \eqref{33}.

Since $\{d\Pi_u(\tilde{X}_{j,i})\}$ is a subspace of $\Gamma(u^{-1}TM)$ on which the energy is concave. So $E_u(\cdot,T_\tau)$ has a unique maximum on $\bar{B}^k$. Since $\{(v_0,\mathcal{T}_{\tau_0}),\{v_j\}\}$ is a collection of harmonic maps, so for $E_{v_0}(\cdot,\mathcal{T}_{\tau_0}):\bar{B}^{k_0}\to[0,\infty)$ defined as
\[E_{v_0}(s_0,\mathcal{T}_{\tau_0})=E(\Pi\circ\big(v_0+\sum_{i=1}^{k_0}s_{0,i}X_{0,i}\big),\mathcal{T}_{\tau_0}),\quad s_0=(s_{0,1},...,s_{0,k_0})\in\bar{B}^{k_0},\]
we know that $E_{v_0}(\cdot,\mathcal{T}_{\tau_0})$ has a unique maximum at $0\in\bar{B}^{k_0}$. Similarly, for $E_{v_j}(\cdot,\mathbb{C}):\bar{B}^{k_j}\to[0,\infty)$ defined as 
\[E_{v_j}(s_j,\mathbb{C})=E(\Pi\circ\big(v_j+\sum_{i=1}^{k_j}s_{j,i}X_{j,i}\big),\mathbb{C}),\quad s_j=(s_{j,1},...,s_{j,k_j})\in\bar{B}^{k_j},\:j=1,...,n,\]
$E_{v_j}(\cdot,\mathbb{C})$ has a unique maximum at $0\in\bar{B}^{k_j}$.
We can argue as in Claim \ref{unstabletorus} and show that $|\frac{d}{ds}E(\Pi\circ(u+s\Tilde{X}_{0,i}),T_\tau)-\frac{d}{ds}E(\Pi\circ(v_0+sX_{0,i}),\mathcal{T}_{\tau_0})|$ and $|\frac{d}{ds}E(\Pi\circ(u+s\Tilde{X}_{j,i}),T_\tau)-\frac{d}{ds}E(\Pi\circ(v_j+sX_{j,i}),\mathbb{C})|$, for $j=1,...,n$, can be sufficiently small so that $E_u(\cdot,T_\tau)$ has a unique maximum at $m(u)\in B^k_{c_0/\sqrt{10}}$. 
\end{proof}
\begin{remark}\label{unstable remark}
The family of vector fields $\{\tilde{X}_{j,i}\}$ constructed in Lemma \ref{unstable} depends on $(u,T_\tau)$. However, by Claim \ref{open unstable} there exists $\delta(u)>0$ such that for any $(u',\mathcal{T}_{\tau'})\in W^{1,2}(\mathcal{T}_{\tau'}, M)$ with $|\tau-\tau'|<\delta(u)$ and
\[\int_{\mathcal{T}_{\tau'}}|\nabla u'-\nabla( u\circ\iota_{\tau,\tau'})|^2<\delta(u),\] 
we have $d_B((u',\mathcal{T}_{\tau'}),\{(v_0,\mathcal{T}_{\tau_0}),\{v_j\}_{j=1}^n\})<\epsilon$. Moreover, \eqref{bubble1}, \eqref{bubble2} and \eqref{bubble3} hold with the same domains $B_{r_j}(x_j)$ and conformal automorphisms $D_j$ for both $(u',\mathcal{T}_{\tau'})$ and $(u,\mathcal{T}_{\tau})$. Since $\{\tilde{X}_{j,i}\}$ depends on $B_{r_j}(x_j)$ and $D_j$, we know that the variation of $(u',\mathcal{T}_{\tau'})$ defined as the following  \[\Pi\circ(u'+\tilde{X}_{j,i}\circ\iota_{\tau,\tau'})\]
also satisfies \eqref{33} and \eqref{34}.
\end{remark}
Now we focus on the case where the conformal structure diverges and the limit maps are possibly several harmonic spheres.
Given a collection of harmonic sheres $\{v_j\}_{j=1}^n$, $v_j:S^2\to M$. Assume that $\text{Index}(v_j)=k_j$, by Definition \ref{indexdef}, there exists subspace  $\{X_{j,i}\}_{i=1}^{k_j}\subset\Gamma(v_j^{-1}M)$ on which the index form is negative definite, by rescaling $\{X_{j,i}\}_{i=1}^{k_j}$ (which we still denote by $\{X_{j,i}\}_{i=1}^{k_j}$), we have that for 
\[E_{v_j}(s_j,\mathbb{C})=E(\Pi\circ\big(v_j+\sum_{i=1}^{k_j}s_{j,i}X_{j,i}\big),\mathbb{C}),\quad s_j=(s_{j,1},...,s_{j,k_j})\in\bar{B}^{k_j},\:j=1,...,n,\]
there exist $c_0<1$ such that 
\begin{equation}\label{unstableassumption2}
-\frac{1}{2c_0}\text{Id}< D^2_{s_j}E_{v_j}(s_j,\mathbb{C})< -2c_0\text{Id},\quad\forall s_j\in\bar{B}^{k_j},
\end{equation}
for each $j=1,...,n$.
\begin{lemma}[Unstable lemma for degenerated torus]\label{unstable1}
Let $\{v_j\}_{j=1}^n$, $v_j:S^2\to M$, be given as above. There exists $\epsilon>0$, such that if a map $u:\mathcal{T}_{\tau}\to M$, $\tau\in\mathcal{M}$, $u\in W^{1,2}(T_\tau,M)$ has sufficiently small bubble norm with $\{v_j\}_{j=1}^n$, i.e., \[d_B((u,\mathcal{T}_{\tau}),\{v_j\}_{j=1}^n)<\epsilon.\]
Then we can construct $\{\Tilde{X}_{j,i}\}\subset C^\infty(T_\tau,\mathbb{R}^N)$, so that for the smooth function \[E_u(\cdot,T_\tau):\bar{B}^k\to[0,\infty),\]
which is defined by
\[E_u(s,T_\tau)=E\big(\Pi\circ\big(u+\sum_{i=1,j=1}^{k_j,n}s_{j,i}\Tilde{X}_{j,i}\big),T_\tau\big),\quad s=(s_{1,1},...,s_{j,k_j},s_{j+1,1},...,s_{n,k_n})\in\bar{B}^k,\]
where $k=\sum_{j=1}^nk_j$, the following holds:
\begin{equation}
    -\frac{1}{c_0}\text{Id}\leq D^2_sE_u(s,T_\tau)\leq -c_0\text{Id},\quad\forall s\in\bar{B}^k.
\end{equation}
Moreover, $E_u(\cdot,T_\tau)$ has a unique maximum at $m(u)\in B^k_{c_0/\sqrt{10}}$ and
\begin{equation}
E_u(m(u),T_\tau)-\frac{1}{2c_0}|s-m(u)|^2\leq E_u(s,T_\tau)\leq E_u(m(u),T_\tau)-\frac{c_0}{2}|s-m(u)|^2,  \end{equation}
for all $s\in\bar{B}^k.$
\end{lemma}
\begin{proof}
From Definition \ref{bubblenorm}, we know that $d_B((u,\mathcal{T}_{\tau}),\{v_j\}_{j=1}^n)<\epsilon$ implies the following:
   \begin{equation}\label{68}
        \begin{split}
        \Big(\int_{T_\tau\setminus\bigcup_{j}B_{r_j}(x_j)}|\nabla u-&\nabla (v_1\circ D_1)|^2\Big)^{1/2}\\
            &+\Big(\int_{S^1\times\mathbb{R}\setminus (T_\tau\setminus\bigcup_jB_{r_j}(x_j)))}|\nabla (v_1\circ D_1)|^2\Big)^{1/2}<\epsilon/3,
        \end{split}
    \end{equation}
    \begin{equation}\label{69}
        \sum_{j=1}^n\Big(\int_{\Omega_j}|\nabla u-\nabla(v_j\circ D_j)|^2\Big)^{1/2}+\Big(\int_{\mathbb{C}\setminus\Omega_j}|\nabla(v_j\circ D_j)|^2\Big)^{1/2}<\epsilon/3,
    \end{equation}
    \begin{equation}\label{70}
        \Big(\int_{\bigcup_j\{B_{r_j}(x_j)\setminus B_{r_j^2}(x_j)\}}|\nabla u|^2\Big)^{1/2}<\epsilon/3.
    \end{equation}
So we can argue as Lemma \ref{unstable}, use cutoff functions to construct $\tilde{X}_{j,i}$. For the bubble map $\{v_j\}_{j=2}^n$, \eqref{69} and \eqref{70} give us the same condition as \eqref{3} and \eqref{4}, so we can define cutoff functions as we did in Lemma \ref{unstable} and construct $\{\tilde{X}_{j,i}\}$ for $j\geq 2$, $i\in\{1,...,k_j\}$. The main difference is \eqref{68}. We use $(t,\theta)$ as parameters on $\mathcal{T}_{\tau}$. Let $\theta'=\text{arg}(\tau)$, and let $z'=t+\sqrt{-1}\theta=e^{-\sqrt{-1}(\frac{\pi}{2}-\theta')}z$ be another conformal parameter system on $\mathcal{T}_{\tau}$. We can conformally expand the torus such that the length of the circle of parameter $\theta$ is one, and we denote the length of the parameter $t$ by $2L$. So we can regard $u$ as a map defined on $S^1\times[-L,L]$. Recall that $S^2\setminus\{$north pole, south pole$\}\simeq\mathbb{C}\setminus\{0\}\simeq S^1\times\mathbb{R}$ by conformal transformation. So we may regard the harmonic sphere $v_1$ as a map defined on $S^1\times\mathbb{R}$. For $j\geq 2$, we can define $\zeta_j$ similar to \eqref{59} with $S^1\times\mathbb{R}$ instead of $\mathbb{C}$, i.e., $\zeta_j$ has value $1$ on $S^1\times\mathbb{R}\setminus D_1\circ B_{r_j}(x_j)$
\[\int_{S^1\times[-L,L]}|\nabla(\zeta_j\circ D_1)|^2=\int_{B_{r_j}(x_j)\setminus B_{r_j^2}(x_j)}|\nabla(\zeta_j\circ D_1)|^2<-\frac{2\pi}{\log r},\]
for some chosen $r$ which can be made sufficiently small when the energy of $v_1\circ D_1$ on the annuli $B_{r_j}(x_j)\setminus B_{r_j^2}(x_j)$ is sufficiently small. (See how we choose $\epsilon_3$ \eqref{e_3} in Lemma \ref{unstable}.) For $j\geq 2$, $\zeta_j$ is constructed to \textit{exclude} the bubble map $v_j$ on $v_1$, now we construct the cutoff function $\zeta_1$ for $v_1$. \eqref{68} implies that
\begin{equation}
    \begin{split}
        \int_{S^1\times\mathbb{R}\setminus (T_\tau\setminus\bigcup_jB_{r_j}(x_j)))}|\nabla (v_1\circ D_1)|^2&\\
        =\int_{S^1\times\mathbb{R}\setminus T_\tau}|\nabla(v_1\circ D_1)|^2&+\int_{\bigcup_jB_{r_j}(x_j)}|\nabla(v_1\circ D_1)|^2
        <\epsilon^2/9.
    \end{split}
\end{equation}
So we have 
\begin{equation}\label{72}
\int_{S^1\times\mathbb{R}\setminus T_\tau}|\nabla(v_1\circ D_1)|^2=\int_{D_1\circ\{S^1\times\mathbb{R}\setminus T_\tau\}}|\nabla v_1|^2<\epsilon^2/9.
\end{equation}
We choose $r$ so that $-\frac{2\pi}{\log r^{1/2}}<\varepsilon$ for some given $\varepsilon>0$. Let $\tilde{v}_1=v_1 \circ e^z\circ\Pi_S$, so that  $\tilde{v}_1$ is defined on $S^2\setminus\{$north pole, south pole$\}$. By the removable singularity property of harmonic map we can extend $\tilde{v}$ to be defined on $S^2$. Let $\tilde{r}$ denote the radius of the geodesic ball $\Pi^{-1}_S(B_r(0))\subset S^2$, now we consider the following
\begin{equation}\label{73}
\epsilon_4:=\min\Big\{E(\tilde{v}_1,B_{\rho}(x)),B_{\rho}(x)\subset S^2,\rho\geq\tilde{r}/2\Big\},    
\end{equation}
where $B_\rho(x)$ is the geodesic ball centered at $x$ with radius $\rho$. We know that since 
\[\Pi_S^{-1}\circ (e^z)^{-1}\circ \{S^1\times\mathbb{R}\setminus T_\tau\}=\Pi_S^{-1}\circ (e^z)^{-1}\circ \{S^1\times(-\infty,-L)\cup S^1\times(L,\infty)\}\]
is the two geodesic ball centered at north pole and south pole seperately, we denote them by $B_{\rho_1}(\hat{n})$ and $B_{\rho_2}(\hat{s})$.  If we choose $\epsilon$ so that $\epsilon^2/9<\epsilon_4$, then by \eqref{73} we know that there wouldn't be any geodesic balls with radius larger than $\tilde{r}$ contained in $D_1\circ B_{\rho_1}(\hat{n})$ or $D_1\circ B_{\rho_2}(\hat{s})$. Now we can define $\zeta_j$ which has energy supported in $D_1\circ B_{\rho_1}(\hat{n})$ and $D_1\circ B_{\rho_2}(\hat{s})$, and with energy smaller than the given $\varepsilon$. Thus we can construct $\tilde{X}_{1,i}=(\prod_{j=1}^n\zeta_jX_{1,i})\circ D_1$ and argue similarly as Lemma \ref{unstable}.
\end{proof}
\begin{lemma}\label{remain}
Let $\{(v_0,\mathcal{T}_{\tau_0}),\{v_j\}_{j=1}^n\}$, $\{X_{j,i}\}$, $\epsilon>0$ be given by Lemma \ref{unstable}. For any $u\in W^{1,2}(T_\tau,M)$, if $\epsilon/2<d_B((u,\mathcal{T}_{\tau}),\{(v_0,\mathcal{T}_{\tau_0}),\{v_j\}_{j=1}^n\})<\epsilon$, then there exists a constant $\nu(\epsilon)>0$ such that if 
\begin{equation}
    E\big(\Pi\circ\big(u+\sum_{i=1,j=0}^{k_j,n}s_{j,i}\Tilde{X}_{j,i}\big),T_\tau\big)\leq E(u,T_\tau)+\nu(\epsilon),
\end{equation}
for some $s=(s_{0,1},...,s_{j,k_j},s_{j+1,1},...,s_{n,k_n})\in\bar{B}^k$,
where $\{X_{j,i}\}$ is given by Lemma \ref{unstable}. Then we have
\begin{equation}
    d_B((\Pi\circ\big(u+\sum_{i=1,j=0}^{k_j,n}s_{j,i}\Tilde{X}_{j,i}\big),\mathcal{T}_{\tau}),\{(v_0,\mathcal{T}_{\tau_0}),\{v_j\}_{j=1}^n\})>\nu(\epsilon).
\end{equation}
\end{lemma}
\begin{proof}
We argue by contradiction and assume that there's a sequence of maps $\{u_m\}_{m\in\mathbb{N}}$, $u_m:\mathcal{T}_{\tau_m}\to M$, with $\epsilon/2<d_B((u_m,\mathcal{T}_{\tau_m}),\{(v_0,\mathcal{T}_{\tau_0}),\{v_j\}_{j=1}^n\})<\epsilon$,
\begin{equation}\label{78}
    E\big(\Pi\circ\big(u_m+\sum_{i=1,j=0}^{k_j,n}s^m_{j,i}\Tilde{X}^m_{j,i}\big),\mathcal{T}_{\tau_m}\big)\leq E(u,\mathcal{T}_{\tau_m})+\frac{1}{m},
\end{equation}
for some $s^m=(s^m_{0,1},...,s^m_{j,k_j},s^m_{j+1,1},...,s^m_{n,k_n})\in\bar{B}^k$,
where $\{\tilde{X}^m_{j,i}\}$ is given by Lemma \ref{unstable} for each $u_m$, and
\begin{equation}\label{79}
d_B((\Pi\circ\big(u_m+\sum_{i=1,j=0}^{k_j,n}s^m_{j,i}\Tilde{X}^m_{j,i}\big)
,\mathcal{T}_{\tau_m}),\{(v_0,\mathcal{T}_{\tau_0}),\{v_j\}_{j=1}^n\})\to 0,\text{ as }m\to\infty.
\end{equation}
\eqref{79} implies that for each $u_m$ we can find $B_{r_{j,m}}(x_{j,m})\subset \mathcal{T}_{\tau_m}$ and $\{D_0^m\}\subset S^1\times S^1$, $\{D_j^m\}\subset\text{PSL}(2,\mathbb{C})$, such that 
\begin{equation}\label{80}
\int_{\mathcal{T}_{\tau_m}\setminus\bigcup_{j}B_{r_{j,m}}(x_{j,m})}|\nabla(\Pi\circ(u_m+\sum_{i=1}^{k_0}s^m_{0,i}\Tilde{X}^m_{0,i}))-\nabla (v_0\circ D^m_0\circ\iota_{\tau_m,\tau_0})|^2\to 0,
\end{equation}
\begin{equation}\label{81}
\int_{\Omega^m_j}|\nabla(\Pi\circ(u_m+\sum_{i=1}^{k_j}s^m_{j,i}\Tilde{X}^m_{j,i}))-\nabla(v_j\circ D_j^m)|^2\to 0,
\end{equation}
where $\Omega^m_j=B_{r_{j,m}^2}(x_{j,m})\setminus\Big\{\bigcup B_{r_{i,m}}(x_{i,m}),\: B_{r_{i,m}}(x_{i,m})\subset B_{r_{j,m}^2}(x_{j,m})\Big\},$ and
\begin{equation}\label{82}
\int_{\bigcup_jB_{r_{j,m}}(x_{j,m})}|\nabla (v_0\circ D^m_0\circ\iota_{\tau,\tau_0})|^2+\int_{\mathbb{C}\setminus\Omega^m_j}|\nabla(v_j\circ D_j^m)|^2\to 0,
\end{equation}
\begin{equation}\label{83}
\int_{\bigcup_j\{B_{r_{j,m}}(x_{j,m})\setminus B_{r_{j,m}^2}(x_{j,m})\}}|\nabla\Pi\circ\big(u_m+\sum_{i=1,j=0}^{k_j,n}s^m_{j,i}\Tilde{X}^m_{j,i}\big)|^2\to 0,
\end{equation}
as $m\to 0$. By how we construct $\Tilde{X}^m_{j,i}\in C^\infty(\mathcal{T}_{\tau_m},M)$ in Lemma \ref{unstable}, we know \eqref{82} implies that the energy of $\Tilde{X}^m_{j,i}$ tends to $0$ on $\bigcup_j\{B_{r_{j,m}}(x_{j,m})\setminus B_{r_{j,m}^2}(x_{j,m})\}$ as $m\to\infty$. (See part 4 in Claim \ref{unstabletorus}.) Thus \eqref{82} and \eqref{83} imply that as $m\to\infty$ we have
\begin{equation}
\int_{\bigcup_j\{B_{r_{j,m}}(x_{j,m})\setminus B_{r_{j,m}^2}(x_{j,m})\}}|\nabla u_m|^2\to 0.
\end{equation}
Assume that $\lim_{m\to\infty}D_0^m=$Id. For almost all $x\in \mathcal{T}_{\tau_0}$, \eqref{80} implies that
\[\lim_{m\to\infty}\Pi\circ\big(u_m\circ\iota_{\tau_0,\tau_m}(x)+\sum_{i=1}^{k_0}(s^m_{0,i}\Tilde{X}^m_{0,i})\circ\iota_{\tau_0,\tau_m}(x)\big)=v_0(x).\]
\eqref{81} implies that
\[\lim_{m\to\infty}\Pi\circ\big(u_m\circ (D^m_j)^{-1}(x)+\sum_{i=1}^{k_j}(s^m_{j,i}\Tilde{X}^m_{j,i}\circ (D^m_j)^{-1}(x))\big)=v_j(x),\quad j\in\{1,...,n\},\]
for almost all $x\in\mathbb{C}$. We recall that $x-\Pi(x)\in T_{\Pi(x)}^\perp M$, then there exist functions $\sigma_0:\mathcal{T}_{\tau_0}\to T_{v_0}^\perp M$, $\sigma_j:\mathbb{C}\to T_{v_j}^\perp M$,
such that 
\[\lim_{m\to\infty}u_m\circ\iota_{\tau_0,\tau_m}(x)+\sum_{i=1}^{k_0}(s^m_{0,i}\Tilde{X}^m_{0,i})\circ\iota_{\tau_0,\tau_m}(x)=v_0(x)+\sigma_0(x),\] 
and 
\[\lim_{m\to\infty}u_m\circ (D^m_j)^{-1}(x)+\sum_{i=1}^{k_0}(s^m_{j,i}\Tilde{X}^m_{j,i})\circ (D^m_j)^{-1}(x)=v_j(x)+\sigma_j(x).\]
Now we consider $Y^m_0=u_m\circ\iota_{\tau_0,\tau_m}-v_0$ and $Y^m_j=u_m\circ(D^m_j)^{-1}-v_j$, since $d\Pi_{v_0}$ is the orthogonal projection of $\mathbb{R}^N$ onto $T_{v_0}M$, so
\begin{equation}\label{85}
\begin{split}
\lim_{m\to\infty}d\Pi_{v_0}(-Y^m_0)(x)&=\lim_{m\to\infty}d\Pi_{v_0}\big(\sum_{i=1}^{k_0}(s^m_{0,i}\Tilde{X}^m_{0,i})\circ\iota_{\tau_0,\tau_m}(x)-\sigma_0(x)\big)\\
&=\lim_{m\to\infty}d\Pi_{v_0}\big(\sum_{i=1}^{k_0}(s^m_{0,i}\Tilde{X}^m_{0,i})\circ\iota_{\tau_0,\tau_m}(x)\big)\\
&=d\Pi_{v_0}\big(\sum_{i=1}^{k_0}(s_{0,i}X_{0,i})(x)\big),
\end{split}    
\end{equation}
where $s_{0,i}:=\lim_{m\to\infty}s^m_{0,i}$, and similarly we have 
\begin{equation}\label{86}
\lim_{m\to\infty}d\Pi_{v_j}(-Y^m_j)(x)=d\Pi_{v_j}\big(\sum_{i=1}^{k_j}(s_{j,i}X_{j,i})(x)\big),   
\end{equation}
where $s_{j,i}:=\lim_{m\to\infty}s^m_{j,i}.$
\begin{equation}\label{87}
\begin{split}
\lim_{m\to\infty}E(u_m)&=\lim_{m\to\infty}\Big\{\sum_{j=1}^n\int_{\Omega_j^m}|\nabla(\Pi\circ(v_j\circ (D^m_j)^{-1
}+(u_m-v_j\circ D^m_j))|^2\\
&+\int_{{\mathcal{T}_{\tau_m}\setminus\bigcup_{j}B_{r_{j,m}}(x_{j,m})}}|\nabla(\Pi\circ(v_0\circ\iota_{\tau_m,\tau_0}+(u_m-v_0\circ\iota_{\tau_m,\tau_0}))|^2\\
&+\int_{\bigcup_j\{B_{r_{j,m}}(x_{j,m})\setminus B_{r_{j,m}^2}(x_{j,m})\}}|\nabla u_m|^2\Big\}\\
&=\lim_{m\to\infty}\Big\{\sum_{j=1}^n\int_{D^m_j\circ\Omega_j^m}|\nabla(\Pi\circ(v_j+Y^m_j)|^2\\
&+\int_{{\mathcal{T}_{\tau_m}\setminus\bigcup_{j}B_{r_{j,m}}(x_{j,m})}}|\nabla(\Pi\circ(v_0+Y^m_0)\circ\iota_{\tau_m,\tau_0}|^2\Big\}\\
&\leq E(\Pi\circ(v_0+Y^m_0),\mathcal{T}_{\tau_0})+\sum_{j=1}^n E(\Pi\circ(v_j+Y^m_j),\mathbb{C})\\
&\leq E(v_0,\mathcal{T}_{\tau_0})+\sum_{j=1}^n E(v_j,\mathbb{C}).
\end{split}
\end{equation}
Since $\{X_{j,i}\}_{i=1}^{k_j}$ is the subspace on $\Gamma(v_j^{-1}TM)$ on which the index form is negative definite, so the last inequality of \eqref{87} follows from \eqref{85} and \eqref{86}, and the equality holds if and only if $s_{0,i}=0$ and $s_{j,i}=0$. On the other hand, by \eqref{78} and \eqref{79} we have the following
\begin{equation}
\begin{split}
E(v_0,\mathcal{T}_{\tau_0})+\sum_{i=1}^nE(v_i,\mathbb{C})=\lim_{m\to\infty}E(\Pi\circ\big(u_m+&\sum_{i=1,j=0}^{k_j,n}s^m_{j,i}\Tilde{X}^m_{j,i}\big),\mathcal{T}_{\tau_m})\\
&\leq\lim_{m\to\infty} E(u_m,\mathcal{T}_{\tau_m}),    
\end{split}
\end{equation}
which forces $s_{0,i}=0$ and $s_{j,i}=0$. Thus we know that \[\lim_{m\to\infty}d_B((u_m,\mathcal{T}_{\tau_m}),\{(v_0,\mathcal{T}_{\tau_0}),\{v_j\}_{j=1}^n\})=0,\]
which is the desired contradiction.
\end{proof}
\section{Min-Max Minimal Torus}
Let $\tilde{\Lambda}=\{(\gamma(t),\mathcal{T}_{\tau(t)});\gamma(t)\in C^0([0,1],C^0\cap W^{1,2}(\mathcal{T}_{\tau(t)},M))\}$, $\tau(t)\in C^0([0,1],\mathcal{M})$, and $\Lambda=\{\gamma(t)\in C^0([0,1],C^0\cap W^{1,2}(T_0,M))\}$. We assume $\gamma(0)$ and $\gamma(1)$ are constant mappings or map the torus to some circles in $M$. And $\tau(0)$, $\tau(1)=\sqrt{-1}$, if mappings on the endpoints are constant mappings and not restrained if not.

For the area functional, we only need to consider the variational problem in the space $\Lambda$, since changing domain metrics will not change the area. But for the energy functional, different conformal structures may lead to different energy, so we have to consider the variational problem in the space $\tilde{\Lambda}$. Fix a homotopically nontrivial path $\beta\in\Lambda$. Let $[\beta]$ be the homotopy class of $\beta$ in $\Lambda$. Since the path $(\beta(t), \mathcal{T}_{\tau_\beta(t)})\in\tilde{\Lambda}$ may have different domains $\mathcal{T}_{\tau_\beta(t)}$ the homotopy equivalence $\alpha\sim\beta$ of  $\alpha(t):\mathcal{T}_{\tau_\alpha(t)}\to M$ and $\beta(t):\mathcal{T}_{\tau_\beta(t)}\to M$ is defined as follows. We can identify $\mathcal{T}_{\tau_\alpha(t)}$ and $\mathcal{T}_{\tau_\beta(t)}$ to $T_0$ by $\iota_{\tau_{\alpha}(t)}$ and 
$\iota_{\tau_{\beta}(t)}$, then we can view $\alpha(t)$ and $\beta(t)$ as mappings defined on the same domain $T_0$ and hence define their homotopy equivalence.
\begin{definition}[Width]\label{width}
Let 
\[W=\inf_{\gamma\in[\beta]}\max_{t\in[0,1]}\text{Area}(\gamma(t)),\] and 
\[W_E=\inf\Big\{\max_{t\in[0,1]}E(\gamma,\mathcal{T}_{\tau(t)});(\gamma(t),\mathcal{T}_{\tau(t)})\in[\beta(t),\mathcal{T}_{\tau_\beta(t)}]\Big\}.\]
\end{definition}
\begin{claim}\cite[Remark 3.2]{minimaltorus}
$W=W_E$.
\end{claim}
\begin{theorem}\cite[Theorem 1.1]{minimaltorus}
For any homotopically nontrivial path $\beta\in\Lambda$, if $W>0$, there exists a sequence $(\rho_l(t),\mathcal{T}_{\tau_l(t)})\in[\beta]$, with $\max_{t\in[0,1]}E(\rho_l(t),\mathcal{T}_{\tau_l(t)})\to W$, and $\forall\epsilon>0$, there exist $N>0$ and $\delta>0$ such that if $l>N$, then for any $t\in(0,1)$ satisfying
\[E(\rho_l(t),\mathcal{T}_{\tau_l(t)})>W-\delta,\]
there are possibly a conformal harmonic torus $v_0:\mathcal{T}_{\tau_0}\to M$ and finitely many harmonic spheres $v_j:S^2\to M$ such that
\[d_V(\rho_l(t),\cup_j v_j)\leq\epsilon,\]
where $d_V$ means varifold distance. (See \cite{minimaltorus}.)
\end{theorem}
\begin{definition}[Minimizing sequence]
We call $\{(\gamma_l(t),\mathcal{T}_{\tau(t)})\}_{l\in\mathbb{N}}\in\tilde{\Lambda}$ a minimizing sequence if
\[\lim_{l\to\infty}\max_{t\in[0,1]}E(\gamma_l(t),\mathcal{T}_{\tau(t)})=W.\]
\end{definition}
\begin{theorem}\label{minmaxtorus}
For any minimizing sequence $\{(\gamma_l(t),\mathcal{T}_{\tau(t)})\}_{l\in\mathbb{N}}$. There exist $\{i_j\}\to\infty$ and $t_{i_j}\in[0,1]$ such that one of the following case holds:
\begin{enumerate}
    \item there exist a conformal harmonic torus $v_0:\mathcal{T}_{\tau_0}\to M$ and finitely many harmonic spheres $v_j:S^2\to M$ such that
    \[d_B((\gamma_{i_j}(t_{i_j}),\mathcal{T}_{\tau_{i_j}(t)}),\{(v_0,\mathcal{T}_{\tau_0},\{v_j\})\})\to 0,\]
    \item there exist finitely many harmonic spheres $v_j:S^2\to M$ such that
    \[d_B((\gamma_{i_j}(t_{i_j}),\mathcal{T}_{\tau_{i_j}(t)}),\{v_j\})\to 0.\]
\end{enumerate}
\end{theorem}
\section{Deformation Theorem}
The metric $g$ is called \textit{bumpy} if there's no prime compact parametrized minimal surface $f:\Sigma\to M$ with branch points, nontransversal crossings or intersections, and
non-trivial Jacobi field other than the Jacobi field generated by the group of conformal automorphisms of $\Sigma.$ (See Definition \ref{bumpy metric}.) It's proved by Moore \cite{JD} that bumpy metrics are generic in the Baire sense. Let $\{(v_0,\mathcal{T}_{\tau_0}),\{v_j\}_{j=1}^n\}$, $\{X_{j,i}\}$ be given as in section 3. Moreover, we assume that the sum of energy is equivalent to width (see Definition \ref{width}), i.e., $E(v_0,\mathcal{T}_{\tau_0})+\sum_{j=1}^nE(v_j,\mathbb{C})=W$. Suppose $(M,g)$ is a bumpy metric and the given harmonic torus $(v_0,\mathcal{T}_{\tau_0})$ is conformal so its image is minimal. Since $g$ is bumpy, $v_j$ has only transversal intersections, so there exists a positive constant $C\leq 1$ such that $\int_{\Omega}|\nabla v_j-\nabla v_i|^2>C^2\int_{\Omega}|\nabla v_j|^2$. Now we consider $\gamma(t)\in C^0([0,1],C^0\cap W^{1,2}(\mathcal{T}_{\tau(t)},M))$, and $\tau(t)\in C^0([0,1],\mathcal{M})$. Let \[\mathcal{I}_\epsilon=\{t\in[0,1], d_B((\gamma(t),\mathcal{T}_{\tau(t)}),\{(v_0,\mathcal{T}_{\tau_0}),\{v_j\}_{j=1}^n\})<\epsilon\},\]
where $\epsilon>0$ is the constant given in Lemma \ref{unstable}. We assume that $\epsilon>0$ is smaller than the energy lower bound for a nontrivial harmonic map in $M$.
\begin{claim}\label{continuous}
For $t\in\mathcal{I}_{C\epsilon/10}$, where $C$ is a constant which depends on $\{(v_0,\mathcal{T}_{\tau_0}),\{v_j\}_{j=1}^n\}$ and $\epsilon>0$ is the constant given in Lemma \ref{unstable}. We can choose $\{X_{j,i}(t)\}\subset C^{\infty}(\mathcal{T}_{\tau(t)},\mathbb{R}^N)$ such that it varies continuously with respect to $t$, and the variation of $\gamma(t)$ defined as the following
\[\Pi\circ\big(\gamma(t)+\sum_{i=1,j=0}^{k_j,n}s_{j,i}X_{j,i}(t)\big),\quad s=(s_{0,1},...,s_{j,k_j},s_{j+1,1},...,s_{n,k_n})\in\bar{B}^k,\]
satisfies \eqref{33} and \eqref{34}.
\end{claim}
\begin{proof} 
Lemma \ref{unstable} implies that for each $t\in\mathcal{I}_\epsilon$ there exist $\{X_{j,i}(t)\}\subset C^{\infty}(\mathcal{T}_{\tau(t)},\mathbb{R}^N)$ such that \eqref{33} and \eqref{34} hold. By how we construct $\{X_{j,i}(t)\}$ we know that it depends on the pairwise disjoint domains $\Omega_j(t)$ and conformal automorphisms $D_j(t)$. If we can find domains $\Omega_j(t)$ and conformal automorphisms $D_j(t)$ that vary continuous with respect to $t$, and satisfy the inequalities \eqref{bubble1}, \eqref{bubble2} and \eqref{bubble3}. Then we can construct $\{X_{j,i}(t)\}$ which varies continuously with respect to $t$.

For each $a\in\mathcal{I}$, there exists $\delta(a)>0$ such that $(a-\delta(a),a+\delta(a))\subset\mathcal{I}$ by Claim \ref{open unstable}. We can cover $\mathcal{I}$ with countably many such open intervals. Moreover, after discarding some of the intervals, we can arrange that each $t\in\mathcal{I}$ is in at least one open interval, each interval intersects at most two others, and the two intervals intersecting the same interval do not intersect each other.

We assume that there's only one bubble map, i.e., $\{(v_0,\mathcal{T}_{\tau_0}),\{v_j\}_{j=1}^n\}=\{(v_0,\mathcal{T}_{\tau_0}),v_1\}$ and we choose some  $t\in(a-\delta(a),a+\delta(a))\cap(b-\delta(b),b+\delta(b))$. There exist $\Omega_a,\Omega_b\subset \mathcal{T}_{\tau(t)}$, $D_a$, $D_b\in$PSL$(2,\mathbb{C})$, and $D_0^a, D_0^b\in S^1\times S^1$ such that \eqref{bubble1}, \eqref{bubble2} and \eqref{bubble3} hold. It's easy to see that $\Omega_a\cap\Omega_b\neq\emptyset$. If not, then
\[
\begin{split}
\Big(\int_{\Omega_a}|\nabla(v_1\circ D_a)-\nabla v_0|^2\Big)^{1/2}&<\Big(\int_{\Omega_a}|\nabla(v_1\circ D_a)-\nabla \gamma(t)|^2\Big)^{1/2}\\
&+\Big(\int_{\Omega_a}|\nabla\gamma(t)-v_0|^2\Big)^{1/2}\\
&<C\epsilon/5.
\end{split}
\]
However, since we assume that $\epsilon$ is smaller than the lower bound of energy for nontrivial harmonic map, we have
\[\Big(\int_{\Omega_a}|\nabla (v_1\circ D_a)|^2\Big)^{1/2}>
\Big(\int_{\mathbb{C}}|\nabla (v_1\circ D_a)|^2\Big)^{1/2}-\Big(\int_{\mathbb{C}\setminus\Omega_a}|\nabla (v_1\circ D_a)|^2\Big)^{1/2}>9\epsilon/10.
\]
Which contradicts that $\int_{\Omega_a}|\nabla v_1\circ D_a-\nabla v_0|^2>C^2\int_{\Omega_a}|\nabla v_a\circ D_a|^2$.
Next, we observe that
\begin{equation*}
    \begin{split}
\Big(\int_{\mathbb{C}}|\nabla(v_1\circ D_a)&-\nabla(v_1\circ D_b)|^2\Big)^{1/2}\leq\Big(\int_{\mathbb{C}\setminus\Omega_a}|\nabla(v_1\circ D_a)|^2\Big)^{1/2}\\
&+\Big(\int_{\mathbb{C}\setminus\Omega_b}|\nabla(v_1\circ D_b)|^2\Big)^{1/2}+\Big(\int_{\Omega_a\cap\Omega_b}|\nabla(v_1\circ D_a)-\nabla(v_1\circ D_b)|^2\Big)^{1/2}\\
&+\Big(\int_{\Omega_a\setminus\Omega_b}|\nabla(v_1\circ D_a)|^2\Big)^{1/2}+\Big(\int_{\Omega_b\setminus\Omega_a}|\nabla(v_1\circ D_b)|^2\Big)^{1/2}.
\end{split}
\end{equation*}
Since $\Big(\int_{\mathbb{C}\setminus\Omega_a}|\nabla(v_1\circ D_a)|^2\Big)^{1/2}+\Big(\int_{\mathbb{C}\setminus\Omega_b}|\nabla(v_1\circ D_b)|^2\Big)^{1/2}$ is bounded by $\epsilon/5$ by \eqref{bubble2}, and 
\[
\begin{split}
\Big(\int_{\Omega_a\cap\Omega_b}|\nabla(v_1\circ D_a)-\nabla(v_1\circ D_b)|^2\Big)^{1/2}&\\
<\Big(\int_{\Omega_a}|\nabla\gamma(t)-\nabla(v_1\circ D_a)|^2\Big)^{1/2}&+\Big(\int_{\Omega_b}|\nabla\gamma(t)-\nabla(v_1\circ D_b)|^2\Big)^{1/2}<\epsilon/5.
\end{split}\]
We're left to deal with $\Big(\int_{\Omega_a\setminus\Omega_b}|\nabla(v_1\circ D_a)|^2\Big)^{1/2}+\Big(\int_{\Omega_b\setminus\Omega_a}|\nabla(v_1\circ D_b)|^2\Big)^{1/2}$. Since 
\begin{equation*}
    \begin{split}
\Big(\int_{\Omega_a\setminus\Omega_b}|\nabla(v_1\circ D_a)-\nabla v_0|^2\Big)^{1/2}&<\Big(\int_{\Omega_a\setminus\Omega_b}|\nabla(v_1\circ D_a)-\nabla \gamma(t)|^2\Big)^{1/2}\\
&+\Big(\int_{\mathbb{C}\setminus\Omega_b}|\nabla\gamma(t)-v_0|^2\Big)^{1/2}\\
&<\epsilon/5,   
    \end{split}
\end{equation*}
which implies $\Big(\int_{\Omega_a\setminus\Omega_b}|\nabla(v_1\circ D_a)|^2\Big)^{1/2}$ is bounded by $\epsilon/5$. Similarly, we can bound $\Big(\int_{\Omega_b\setminus\Omega_a}|\nabla(v_1\circ D_b)|^2\Big)^{1/2}<\epsilon/5$.  So we can bound $\Big(\int_{\mathbb{C}}|\nabla(v_1\circ D_a)-\nabla(v_1\circ D_b)|^2\Big)^{1/2}<\epsilon$, which implies that $\Omega_a$ is close to $\Omega_b$ in PSL$(2,\mathbb{C})$, so we can define a continuous function $D:[a,b]\to$ PSL$(2,\mathbb{C})$ with $D(a)=D_a$ and $D(b)=D_b$, such that
\[\int_{\mathbb{C}\setminus\Omega_a\cap\Omega_b}|\nabla(v_1\circ D(t))|^2<\epsilon,\quad\forall t\in[a,b],\]
and define a continuous function $\Omega:[a,b]\to\mathbb{C}$, $\Omega_a\cap\Omega_b\subseteq\Omega(t)\subseteq\Omega_a\cup\Omega_b$, with $\Omega(a)=\Omega_a$ and $\Omega(b)=\Omega_b$, such that for $\gamma(t)$ the inequality \eqref{bubble2} holds with $\Omega(t)$ and $D(t)\in\text{PSL}(2,\mathbb{C})$. Similarly, we can construct a continuous function $D_0(t)\in S^1\times S^1$ such that the inequality  \eqref{bubble1} holds with $D_0(t)$ and $\mathcal{T}_{\tau(t)}$. Then we can construct $\{X_{j,i}(t)\}$ which varies continuously with respect to $t$.
\end{proof}
For a given $\gamma(t)\in C^0([0,1],C^0\cap W^{1,2}(\mathcal{T}_{\tau(t)},M))$, and $\tau(t)\in C^0([0,1],\mathcal{M})$. Let $I\subset\mathcal{I}_{C\epsilon/10}$ be a closed interval. We can construct $\{X_{j,i}(t)\}\subset C^{\infty}(\mathcal{T}_{\tau(t)},\mathbb{R}^N)$ such that it varies continuously with respect to $t\in I$ By Claim \ref{continuous}. The variation of $\gamma(t)$ defined as the following
\[\Pi\circ\big(\gamma(t)+\sum_{i=1,j=0}^{k_j,n}s_{j,i}X_{j,i}(t)\big),\quad s=(s_{0,1},...,s_{j,k_j},s_{j+1,1},...,s_{n,k_n})\in\bar{B}^k,\]
and we define the smooth function $E_{\gamma(t)}(\cdot,\mathcal{T}_{\tau(t)}):\Bar{B}^k\to[0,\infty)$ by
\[E_{\gamma(t)}(s,\mathcal{T}_{\tau(t)})=E\big(\Pi\circ\big(\gamma(t)+\sum_{i=1,j=0}^{k_j,n}s_{j,i}X_{j,i}(t)\big),\mathcal{T}_{\tau(t)}\big),\]
where $s=(s_{0,1},...,s_{j,k_j},s_{j+1,1},...,s_{n,k_n})\in\bar{B}^k$. By Claim \ref{continuous} we know that $E_{\gamma(t)}(\cdot,\mathcal{T}_{\tau(t)})$ satisfies \eqref{33} and \eqref{34}. Moreover, the maximum $m(\gamma(t))\in B_{c_0/\sqrt{10}}^k$ of $E_{\gamma(t)}(\cdot,\mathcal{T}_{\tau(t)})$ varies continuously with respect to $t$.
We consider the one-parameter flow
\begin{align*}
\{\phi_t(\cdot,x)\}_{x\geq 0}&\in\text{Diff}(\bar{B}^k)\\
 \phi_t(\cdot,\cdot):\bar{B}^k&\times [0,\infty)\to \bar{B}^k,   
\end{align*}
generated by the vector field:
\begin{equation}
s\mapsto -(1-|s|^2)\nabla E_{\gamma(t)}(s,\mathcal{T}_{\tau(t)}),\:s\in\bar{B}^k. 
\end{equation}
\begin{claim}\label{flow}		
For all $\kappa<\frac{1}{4}$, there is $T$ depending on $\big\{\{(v_0,\mathcal{T}_{\tau_0}),\{v_j\}_{j=1}^n\}, \{X_{j,i}\},\kappa,\epsilon\big\}$ so that for any $t\in I\subset\mathcal{I}_{C\epsilon/10}$, and $\chi\in \bar{B}^k$ with $|\chi-m(\gamma(t))|\geq\kappa$ we have:
\begin{equation}\label{decrease flow}
E_{\gamma(t)}(\phi_t(\chi,T),\mathcal{T}_{\tau(t)})<E_{\gamma(t)}(0,\mathcal{T}_{\tau(t)})-\frac{c_0}{10}.
\end{equation}
\end{claim}
\begin{proof}
By $m(\gamma(t))\in B^k_\frac{c_0}{\sqrt{10}}(0)$ and \eqref{34}, we have:
\begin{equation}\label{c}
\sup_{s\in\bar{B}^k}E_{\gamma(t)}(s,\mathcal{T}_{\tau(t)})=E_{\gamma(t)}(m(\gamma(t)),\mathcal{T}_{\tau(t)})\leq E_{\gamma(t)}(0,\mathcal{T}_{\tau(t)})+\frac{c_0}{20},\:\forall t\in I.
\end{equation}
So, to prove (\ref{decrease flow}), it suffices to show the existence of $T$ such that
\[|\chi-m(\gamma(t))|\geq\kappa\implies E_{\gamma(t)}(\phi_t(\chi,T),\mathcal{T}_{\tau(t)})<\sup_{s\in\bar{B}^k}E_{\gamma(t)}(s,\mathcal{T}_{\tau(t)})-\frac{c_0}{5}.\]
We argue by contradiction and assume that there exists a constant  $\frac{1}{4}>\kappa>0$, a sequence $\{t_l\}_{l\in\mathbb{N}}\subset I$, and $\{s_l\}_{l\in\mathbb{N}}\subset\bar{B}^k$ with $|s_l-m(\gamma(t_l))|\geq\kappa$ such that
\begin{equation}\label{abc}
E_{\gamma(t_l)}(\phi_{t_l}(s_l,l),\mathcal{T}_{\tau(t_l)})\geq E_{\gamma(t_l)}(0,\mathcal{T}_{\tau(t_l)})-\frac{c}{10}.
\end{equation}
Combining (\ref{abc}) with (\ref{c}) we have $E_{\gamma(t_l)}(\phi_{t_l}(s_l,l),\mathcal{T}_{\tau(t_l)})\geq E_{\gamma(t_l)}(m(\gamma(t_l)),\mathcal{T}_{\tau(t_l)})-\frac{c_0}{5}$. Since $\phi_{t}(\cdot,\cdot)$ is an energy decreasing flow, we have
\[E_{\gamma(t_l)}(\phi_{t_l}(s_l,x),\mathcal{T}_{\tau(t_l)})\geq E_{\gamma(t_l)}(\phi_{t_l}(s_l,l),\mathcal{T}_{\tau(t_l)})\geq E_{\gamma(t_l)}(m_j(t_l),\mathcal{T}_{\tau(t_l)})-\frac{c_0}{5},\:\forall 0\leq x\leq l.\]
Since both $I$ and $\bar{B}^k$ are compact, we obtain subsequential limits $t'\in I$ and $s\in\bar{B}^k$ with
\begin{equation}\label{88}
E_{\gamma(t')}(\phi_{t'}(s,x),\mathcal{T}_{\tau(t')})\geq \sup_{|v|\leq 1}E_{\gamma(t')}(v,\mathcal{T}_{\tau(t')})-\frac{c_0}{5},\quad\forall x\geq 0.
\end{equation}
Since $\gamma(t)\mapsto m(\gamma(t))$ is a continuous map, $|s_l-m(\gamma(t_l))|\geq\kappa$ implies $|s-m(\gamma(t'))|\geq\kappa$. Thus we have $\lim_{x\to\infty}|\phi_{t'}(s,x)|=1$ and thus we deduce from the equation (\ref{88}) that
\begin{equation}
\sup_{|v|=1}E_{\gamma(t')}(v,\mathcal{T}_{\tau(t')})\geq\sup_{|v|\leq 1}E_{\gamma(t')}(v,\mathcal{T}_{\tau(t')})-\frac{c_0}{5}.
\end{equation}
On the other hand,  $m(\gamma(t'))\in B^k_\frac{c_0}{\sqrt{10}}(0)$ implies $|v-m(\gamma(t'))|>2/3$ for all $v\in\bar{B}^k$ with $|v|=1$. Hence, we have
\[\sup_{|v|=1}E_{\gamma(t')}(v,\mathcal{T}_{\tau(t')})\leq\sup_{|v|\leq 1}E_{\gamma(t')}(v,\mathcal{T}_{\tau(t')})-\frac{c_0}{2}\Big(\frac{2}{3}\Big)^2<\sup_{|v|\leq 1}E_{\gamma(t')}(v,\mathcal{T}_{\tau(t')})-\frac{c_0}{5},\]
			which gives us the desired contradiction.
		\end{proof}

\begin{theorem}[Deformation Theorem]\label{deformation}
Let $(M,g)$, $\{(v_0,\mathcal{T}_{\tau_0}),\{v_j\}\}$ and $\{X_{j,i}\}$ be given as above. If $\sum_{j=0}^n\text{Index}(v_j)=k\geq 2,$ then for a given minimizing sequence $\{(\gamma_l,\mathcal{T}_{\tau_l(t)})\}_{l\in\mathbb{N}}\subset C^0([0,1],C^0\cap W^{1,2}(\mathcal{T}_{\tau_l(t)},M))$, and $\tau_l(t)\in C^0([0,1],\mathcal{M})$. There exists a sequence $\{(\gamma'_l,\mathcal{T}_{\tau_l(t)})\}_{l\in\mathbb{N}}\subset C^0([0,1],C^0\cap W^{1,2}(\mathcal{T}_{\tau_l(t)},M))$, $l_0\in\mathbb{N}$ and $\bar{\epsilon}>0$ such that
\begin{enumerate}
    \item $(\gamma'_l,\mathcal{T}_{\tau_l(t)})$ is homotopic to $(\gamma'_l,\mathcal{T}_{\tau_l(t)})$,
    \item 
    $d_B((\gamma'_l(t),\mathcal{T}_{\tau_l(t)}),\{(v_0,\mathcal{T}_{\tau_0}),\{v_j\}\})\geq\bar{\epsilon}$ for all $l\geq l_0,$ $t\in[0,1]$
\end{enumerate}
\end{theorem}
\begin{proof}
Let \[\mathcal{I}_{l,\epsilon}=\{t\in[0,1], d_B((\gamma_l(t),\mathcal{T}_{\tau_l(t)}),\{(v_0,\mathcal{T}_{\tau_0}),\{v_j\}_{j=1}^n\})<C\epsilon/10\},\]
where the constant $C$ is given in Claim \ref{continuous}. According to Claim \ref{continuous}, for $t\in\mathcal{I}_{l,\epsilon}$, we can choose $\{X^l_{j,i}(t)\}\subset C^{\infty}(\mathcal{T}_{\tau_l(t)},\mathbb{R}^N)$ such that it varies continuously with respect to $t$. We define the variation of $\gamma_l(t)$ as the following
\[\Pi\circ\big(\gamma_l(t)+\sum_{i=1,j=0}^{k_j,n}s_{j,i}X^l_{j,i}(t)\big),\quad s=(s_{0,1},...,s_{j,k_j},s_{j+1,1},...,s_{n,k_n})\in\bar{B}^k,\]
and we define the smooth function $E_{\gamma_l(t)}(\cdot,\mathcal{T}_{\tau_l(t)}):\Bar{B}^k\to[0,\infty)$ by
\[E_{\gamma_l(t)}(s,\mathcal{T}_{\tau_l(t)})=E\big(\Pi\circ\big(\gamma_l(t)+\sum_{i=1,j=0}^{k_j,n}s_{j,i}X^l_{j,i}(t)\big),\mathcal{T}_{\tau_l(t)}\big),\]
where $s=(s_{0,1},...,s_{j,k_j},s_{j+1,1},...,s_{n,k_n})\in\bar{B}^k$. By Claim \ref{continuous} we know that $E_{\gamma_l(t)}(\cdot,\mathcal{T}_{\tau_l(t)})$ satisfies \eqref{33} and \eqref{34}. Moreover, the maximum $m(\gamma_l(t))\in B_{c_0/\sqrt{10}}^k$ of $E_{\gamma_l(t)}(\cdot,\mathcal{T}_{\tau_l(t)})$ varies continuously with respect to $t$.
We define a continuous homotopy: $$H'_{l}:\mathcal{I}_{l,\epsilon}\times[0,1]\longrightarrow B^k_{1/2^l}(0),$$
so that
$$H'_l(t,0)=0,\:\text{and}\:\inf_{t\in\mathcal{I}_{l,\epsilon}}|H'_l(t,1)-m(\gamma_l(t))|\geq\kappa_l>0.$$
We are able to define $H_j'$ due to the assumption $\sum_{j=0}^n \text{Index}(v_j)=k\geq 2$. So we can choose a continuous path in $B^k_{1/2^j}(0)$ away from the curve of $m(\gamma_l(t))$. By Claim \ref{flow}, for each $l\in\mathbb{N}$, there exists $T_l\in[0,\infty)$ for $t\in\mathcal{I}_{l,\epsilon}$ such that:
$$E_{\gamma_l(t)}\big(\phi_t^l(H_l'(t,1),T_l),\mathcal{T}_{\tau_l(t)}\big)<E_{\gamma_l(t)}\big(0,\mathcal{T}_{\tau_l(t)}\big)-\frac{c_0}{10},$$
where $\phi_t^l(\cdot,\cdot)$ is the one-parameter flow
\begin{align*}
\{\phi_t^l(\cdot,x)\}_{x\geq 0}&\in\text{Diff}(\bar{B}^k)\\
 \phi_t^l(\cdot,\cdot):\bar{B}^k&\times [0,\infty)\to \bar{B}^k,   
\end{align*}
generated by the vector field:
\begin{equation}
s\mapsto -(1-|s|^2)\nabla E_{\gamma_l(t)}(s,\mathcal{T}_{\tau_l(t)}),\:s\in\bar{B}^k. 
\end{equation}
Let $c_l:[0,1]\longrightarrow[0,1]$ be a cutoff function which is supported in $\mathcal{I}_{l,\epsilon}$, and has value one in $\mathcal{I}_{j,\epsilon/2}$, value zero in $[0,1]\setminus\mathcal{I}_{j,\epsilon}.$
Define:
$$H_l(t,x)=H'_l(t,c_l(t)x),$$
and
\[H_l(t,x)=0\quad\forall t\in [0,1]\setminus\mathcal{I}_{l,\epsilon}.\] 
We now set $s_l:[0,1]\to\bar{B}^k$ to be \[s_l(t)=\phi_t^l(H_l(t,1),c_l(t)T_l),\quad\text{if }t\in\mathcal{I}_{l,\epsilon},\]
and 
\[s_l(t)=0,\quad\text{if }t\in[0,1]\setminus\mathcal{I}_{l,\epsilon}.\]
Now we define $\gamma'_l(t)$ as the following
\[\gamma'_l(t):=\Pi\circ\big(\gamma_l(t)+\sum_{i=1,j=0}^{k_j,n}s^l_{j,i}(t)X^l_{j,i}(t)\big),\]
where $s_l(t)=(s^l_{0,1}(t),...,s^l_{j,k_j}(t),s^l_{j+1,1}(t),...,s^l_{n,k_n}(t))\in\bar{B}^k$. Since $s_l$ is homotopic to the zero map in $\bar{B}^k$, so $\gamma'_{l}(t)$ is homotopic to $\gamma_{l}(t).$

\begin{claim}
There exist $\bar{\epsilon}>0$ and $l_0\in\mathbb{N}$ such that 
\[d_B((\gamma'_l(t),\mathcal{T}_{\tau_l(t)}),\{(v_0,\mathcal{T}_{\tau_0}),\{v_j\}\})\geq\bar{\epsilon}\]
for all $l\geq l_0,$ $t\in[0,1]$.
\end{claim}
\begin{description}
\item[Case 1]$t\in[0,1]\setminus\mathcal{I}_{l,\epsilon}.$ Since we have $\gamma_l(t)=\gamma_l'(t)$ for $t\in[0,1]\setminus\mathcal{I}_{l,\epsilon}$, then 
\[d_B((\gamma'_l(t),\mathcal{T}_{\tau_l(t)}),\{(v_0,\mathcal{T}_{\tau_0}),\{v_j\}\})\geq C\epsilon/10.\]
\item[Case 2]$t\in\mathcal{I}_{l,\epsilon}\setminus\mathcal{I}_{l,\epsilon/2}$. Since $\phi_t^l(\cdot,\cdot)$ is the energy decreasing flow, so we have
\[E_{\gamma_l(t)}\big(\phi_t^l(H_l(t,1),T),\mathcal{T}_{\tau_l(t)}\big)\leq E_{\gamma_l(t)}\big(H_l(t,1),\mathcal{T}_{\tau_l(t)}\big),\quad\forall T\in[0,T_l].\]
Moreover, since $H_l(t,1)\in B^k_{1/2^l}$, there exists $l_1\in\mathbb{N}$ such that
\[E_{\gamma_l(t)}(s_l(t),\mathcal{T}_{\tau_l(t)})\leq E_{\gamma_l(t)}(0,\mathcal{T}_{\tau_l(t)})+\nu(C\epsilon/10),\]
for all $l\geq l_1$. By Lemma \ref{remain}, we have \[d_B((\gamma'_l(t),\mathcal{T}_{\tau_l(t)}),\{(v_0,\mathcal{T}_{\tau_0}),\{v_j\}\})>\nu(C\epsilon/10).\]
\item[Case 3]$t\in\mathcal{I}_{l,\epsilon/2}$. If $d_B((\gamma'_l(t),\mathcal{T}_{\tau_l(t)}),\{(v_0,\mathcal{T}_{\tau_0}),\{v_j\}\})\to 0$ as $l\to\infty$, then we have $E(\gamma'_l(t),\mathcal{T}_{\tau_l(t)})=E(v_0,\mathcal{T}_{\tau_0})+\sum_{j=1}^nE(v_j,\mathbb{C})=W$. By Claim \ref{flow} we have 
\[E_{\gamma_l(t)}(s_l(t),\mathcal{T}_{\tau_l(t)})\leq E_{\gamma_l(t)}(0,\mathcal{T}_{\tau_l(t)})-c_0/10.\]
Since $\{(\gamma_l(t),\mathcal{T}_{\tau_l(t)})\}$ is a minimizing sequence, there exists $l_2\in\mathbb{N}$ such that
\[E_{\gamma_l(t)}(s_l(t),\mathcal{T}_{\tau_l(t)})\leq W-c_0/20,\quad\forall l\geq l_2,\]
which implies that there exists $\tilde{\epsilon}>0$ such that
\[d_B((\gamma'_l(t),\mathcal{T}_{\tau_l(t)}),\{(v_0,\mathcal{T}_{\tau_0}),\{v_j\}\})\geq \tilde{\epsilon},\quad\forall l\geq l_2.\]
Let $l_0:=\max\{l_1,l_2\}$ and $\bar{\epsilon}:=\min\{C\epsilon/10,\nu(C\epsilon/10),\tilde{\epsilon}\}$, we have proved the Theorem.
\end{description}    

\end{proof}
\subsection{Proof of Theorem \ref{main}}
\subsubsection*{Theorem \ref{main}}{\em Let $M$ be a closed Riemannian manifold of dimension at least three and suppose $(M,g)$ is a bumpy metric. For any homotopically nontrivial path $\beta\in\Lambda$, if $W>0$. Then there are possibly a conformal torus $v_0:\mathcal{T}_{\tau_0}\to M$ and finitely many harmonic spheres $\{v_j\}$ such that 
\begin{enumerate}
\item $E(v_0,\mathcal{T}_{\tau_0})+\sum_{j} E(v_j,\mathbb{C})=W,$
\item $\sum_{j} Index(v_j)\leq 1.$
	\end{enumerate}}
	\begin{proof}
Denote by $\mathcal{F}^W$ the collection of equivalent classes of conformal harmonic tori and harmonic spheres. By proposition \ref{countable} and \cite[Proposition B.24]{YS}, $\mathcal{F}^W$ is countable and thus the following set is countable:
\[\mathcal{U}^W=\Big\{\{[(v^1_0,\mathcal{T}_{\tau_0^1})],[v_j^1]\},\{[(v^2_0,\mathcal{T}_{\tau_0^2})],[v_j^2]\},...,\{[(v^n_0,\mathcal{T}_{\tau_0^1})],[v_j^n]\},...\Big\},\]
with $E(v_0^n,\mathcal{T}_{\tau_0^n})+\sum_j E(v_j^n,\mathbb{C})=W$ and $\sum_{j}\text{Index}(v_j^n)\geq 2$ for each $n\in\mathbb{N}$. 
		
Given a minimizing sequence $\{(\gamma_j,\mathcal{T}_{\tau_j(t)})\}_{j\in\mathbb{N}}$, $\gamma_j(\cdot,t):\mathcal{T}_{\tau_j(t)}\to M$, $(\gamma_j,\mathcal{T}_{\tau_j(t)})\in\tilde{\Lambda}$ for each $j$, with $[(\gamma_j,\mathcal{T}_{\tau_j(t)})]=[\beta]$, we abuse the notation here by using $j$ as the index for minimizing sequence and the finite collection of harmonic spheres in $\mathcal{U}^W$. Considering $\{[(v^1_0,\mathcal{T}_{\tau_0^1})],[v_j^1]\}\in\mathcal{U}^W$, and by Theorem \ref{deformation} there exists $\{(\gamma^1_j,\mathcal{T}_{\tau_j(t)})\}_{j\in\mathbb{N}}$ so that
\begin{enumerate}
\item $[(\gamma^1_j,\mathcal{T}_{\tau_j(t)})]=[\beta]$, $\forall j\in\mathbb{N}$,
\item $\{(\gamma^1_j,\mathcal{T}_{\tau_j(t)})\}_{j\in\mathbb{N}}$ is a minimizing sequence,
\item 
there exists $\epsilon_1>0$ and $j_1\in\mathbb{N}$ such that
\[d_B((\gamma_j^1(t),\mathcal{T}_{\tau_j(t)}),\{[(v^1_0,\mathcal{T}_{\tau_0^1})],[v_j^1]\})>\epsilon_1,\quad\forall t\in[0,1],\:\forall j>j_1.\]
		\end{enumerate}
We can apply Theorem \ref{deformation} again with $\{[(v^2_0,\mathcal{T}_{\tau_0^2})],[v_j^2]\}\in\mathcal{U}^W$ and obtain $\{\gamma_j^2,\mathcal{T}_{\tau_j(t)}\}_{j\in\mathbb{N}}$ so that
\begin{enumerate}
\item $[(\gamma^2_j,\mathcal{T}_{\tau_j(t)})]=[\beta]$, $\forall j\in\mathbb{N}$,
\item $\{(\gamma^2_j,\mathcal{T}_{\tau_j(t)})\}_{j\in\mathbb{N}}$ is a minimizing sequence,
\item 
there exist $\epsilon_1,\epsilon_2>0$ and $j_1,j_2\in\mathbb{N}$ such that
\[d_B((\gamma_j^2(t),\mathcal{T}_{\tau_j(t)}),\{[(v^l_0,\mathcal{T}_{\tau_0^1})],[v_l^1]\})>\epsilon_l,\quad\forall t\in[0,1],\:\forall j>j_l,\:l=1,2.\]
\end{enumerate}
Proceeding inductively we can find $\{(\gamma_j^m,\mathcal{T}_{\tau_j(t)})\}_{j\in\mathbb{N}}$ such that
\begin{enumerate}
\item $[(\gamma^m_j,\mathcal{T}_{\tau_j(t)})]=[\beta]$, $\forall j\in\mathbb{N}$,
\item $\{(\gamma^m_j,\mathcal{T}_{\tau_j(t)})\}_{j\in\mathbb{N}}$ is a minimizing sequence,
\item 
there exist $\epsilon_l>0$ and $j_l\in\mathbb{N}$ for $l=1,...,m$ such that
\[d_B((\gamma_j^m(t),\mathcal{T}_{\tau_j(t)}),\{[(v^l_0,\mathcal{T}_{\tau_0^1})],[v_l^1]\})>\epsilon_l,\quad\forall t\in[0,1],\:\forall j>j_l,\:l=1,...,m.\]
\end{enumerate}
We can choose an increasing sequence $p_m>j_m$ such that
\[\max_{t\in[0,1]}E(\gamma^m_{p_m}(\cdot,t),\mathcal{T}_{\tau_{p_m}(t)})\leq W+\frac{1}{m}.\]
The sequence $\{(\gamma^m_{p_m}(\cdot,t),\mathcal{T}_{\tau_{p_m}(t)})\}_{m\in\mathbb{N}}$ is a minimizing sequence, thus by Theorem \ref{minmaxtorus}, there exists a sequence $\{t_m\}_{m\in\mathbb{N}}\subset[0,1]$ such that one of the following case holds:
\begin{enumerate}
    \item there exist a conformal harmonic torus $v_0:\mathcal{T}_{\tau_0}\to M$ and finitely many harmonic spheres $v_j:S^2\to M$ such that
    \[d_B((\gamma_{p_m}^m(t_{m}),\mathcal{T}_{\tau_{p_m}(t)}),\{(v_0,\mathcal{T}_{\tau_0},\{v_j\})\})\to 0,\:m\to\infty\]
    \item there exist finitely many harmonic spheres $v_j:S^2\to M$ such that
    \[d_B((\gamma_{p_m}^m(t_{m}),\mathcal{T}_{\tau_{p_m}(t)}),\{v_j\})\to 0,\:m\to\infty.\]
\end{enumerate}
In either case we have 
\begin{enumerate}
\item $E(v_0,\mathcal{T}_{\tau_0})+\sum_{j} E(v_j,\mathbb{C})=W,$
\item $\sum_{j}\text{Index}(v_i)\leq 1.$
\end{enumerate}
\end{proof}
\appendix
\section{Countability of conformal Harmonic Torus}
We are going to prove in this section that for a generic choice of metric the space of conformal harmonic torus with bounded energy is countable.
\subsection{Jocabi Field}
\begin{theorem}[A Priori Estimate \cite{SU} Main Estimate 3.2]\label{regularity}
		Given $\Sigma$ and $M$, there exist $\epsilon_{su}>0$ and $\rho>0$ such that if $r_0<\rho$, $u:\Sigma\to M$ is harmonic and
		\[\int_{B_{r_0}^{\Sigma}(y)}|\nabla u|^2<\delta\epsilon_{su},\]
		then
		\[|\nabla u|^2(y)\leq\frac{\delta}{r_0^2}.\]
	\end{theorem}
\begin{remark}
Once we know that $\nabla u \in L^{\infty}_{\text{loc}}$, it then follows from equation \eqref{EL2} that $\Delta u\in L^{\infty}_{\text{loc}}$, which implies by standard estimates on the inverse of the Laplacian that $u\in W^{2,p}_{\text{loc}}$ for all $p<\infty$. Hence we deduce that $\Delta u\in W^{1,p}_{\text{loc}}$ and hence $u\in W^{3,p}_{\text{loc}}$ for all $p>0$. We can then repeat this argument to show that $u\in W^{r,p}_{\text{loc}}$, $\forall r$, and so the smoothness of the solution follows. 
\end{remark}
	
\begin{definition}[tension field] We call the tension field of $f$ to be the following
\[\tau(f):=\text{trace}\nabla df.\]
\end{definition}
\begin{remark}
Intrinsically, the Euler-Lagrange equation for energy is
\begin{equation}
    \tau(f)=0.
\end{equation}
\end{remark}
\begin{definition}[Jocabi Field]
For a harmonic map $f:\Sigma\to M$, the Jacobi operator $\mathcal{J}_f:\Gamma(f^{-1}TM)\to\Gamma(f^{-1}TM)$ is defined as following
\begin{equation}\label{jocabi}
\mathcal{J}_f(V):=-\Delta V-\text{trace}_\Sigma R^M(V,df)df,
\end{equation}
where $\Delta$ is the Laplacian on sections of $f^{-1}TM$ given in local coordinates on $\Sigma$ by 
\[\Delta=h^{\alpha\beta}(f^*\nabla^M)_{\frac{\partial}{\partial x^\alpha}}(f^*\nabla^M)_{\frac{\partial}{\partial x^\beta}},\]
so we have that
\[I(V,W)=\int_\Sigma\langle\mathcal{J}_f(V),W\rangle.\]
We call $V\in\Gamma(f^{-1}TM)$ a Jacobi field for $f$ if $\mathcal{J}_f(V)=0.$
\end{definition}
Let
\begin{equation}\label{2.16}
    \begin{split}
        &f_{st}(x)=f(x,s,t),\\
        &f:\Sigma\times(-\epsilon,\epsilon)\times(-\epsilon,\epsilon)\to M
    \end{split}
\end{equation}
be a smooth family of maps between Riemannian manifolds of finite energy. $\Sigma$ may have nonempty boundary, in which case we require $f(x,s,t)=f(x,0,0)$ for all $x\in\partial\Sigma$ and all $s,t$.
\begin{proposition}\label{prop7}
For a smooth family of maps $f_{st}:\Sigma\to M$ defined as \eqref{2.16}, with
\begin{align*}
    V:=&\frac{\partial}{\partial s}\at[\Big]{s=t=0}f_{st}\\
    W:=&\frac{\partial}{\partial t}\at[\Big]{s=t=0}f_{st}.
\end{align*}
Let $f_{00}=f$ be a smooth harmonic map. We have the Jacobi operator $\mathcal{J}_f$ defined as (\ref{jocabi}) to be the following:
\[\mathcal{J}_f(V)=-\frac{\partial}{\partial s}\at[\Big]{s=0}\tau(f_{s0}).\]
\end{proposition}
\begin{proof}
The proposition follows from the computation below:
\begin{align*}
\frac{\partial^2}{\partial s\partial t}\at[\Big]{s=t=0}E(f_{st})=\Big(&\frac{\partial}{\partial s}\at[\Big]{s=0}\int_M\langle df_{st},\nabla_{\frac{\partial}{\partial t}}df_{st}\rangle\Big)\at[\Big]{t=0}\\
    =&-\frac{\partial}{\partial s}\at[\Big]{s=0}\int_M\langle \tau(f_{s0}),\frac{\partial}{\partial t}\at[\Big]{t=0}f_{st}\rangle\\
    =&\int_M\langle-\frac{\partial}{\partial s}\at[\Big]{s=0}\tau(f_{s0}),\frac{\partial}{\partial t}\at[\Big]{s=t=0}f_{st}\rangle +\int_M\langle\tau(f_{00}),-\frac{\partial^2}{\partial s\partial t}\at[\Big]{s=t=0}f_{st}\rangle \\
    =&\int_M\langle-\frac{\partial}{\partial s}\at[\Big]{s=0}\tau(f_{s0}),W\rangle \\
    =&\int_M\langle\mathcal{J}_f(V),W\rangle.
\end{align*}
The second to last equality follows by $f_{00}$ is harmonic. Then it implies that 
\[\mathcal{J}_f(V)=-\frac{\partial}{\partial s}\at[\Big]{s=0}\tau(f_{s0}).\]
\end{proof}
Given a sequence of harmonic torus $f_i:\mathcal{T}_{\tau_i}\to M$ converges to a  harmonic torus $u:T_\tau\to M$ in the following way: 
\begin{equation}\label{101}
\tau_i\to\tau,\int_{\mathcal{T}_{\tau}}|\nabla(f_i\circ\iota_{\tau,\tau_i})-\nabla u|^2\to 0,    
\end{equation}
then we can use the \textit{difference} of them to generate a Jacobi field of $u$ (possibly trivial or just the tangential Jacobi fields generated by the conformal automorphisms of the domain).
\begin{align*}
\tau(f_{i}\circ\iota_{\tau,\tau_i})-\tau(u)&=\int_{0}^1\frac{\partial}{\partial s}\at[\Big]{s=t}\tau(u+s(f_{i}\circ\iota_{\tau,\tau_i}-u))dt\\  
&=\frac{\partial}{\partial t}\at[\Big]{t=t_i}\tau(u+t(f_{i}\circ\iota_{\tau,\tau_i}-u)),
\end{align*}
so we define $w$ to be the following
\begin{equation}\label{102}
    w(x):=\lim_{i\to\infty}\frac{f_{i}\circ\iota_{\tau,\tau_i}(x)-u(x)}{\max_{y\in \mathcal{T}_{\tau}}|f_{i}\circ\iota_{\tau,\tau_i}(y)-u(y)|},\quad\forall x\in T_\tau,
\end{equation}
and we observe that if $f_{i}\circ\iota_{\tau,\tau_i}\neq u$ then $w$ is well defined. Thus by Proposition \ref{prop7} and \eqref{101} we get that $\mathcal{J}_u(w)=0.$
\subsection{Bumpy Metric Theorem for Minimal Tori}
Now we define generic metric for a specific metric space. For $k\in\mathbb{N}$, the space of $L^2_k$ Riemannian metrics on $M$ simply denotes an open set of the Hilbert space of $L^2_k$-sections of the second symmetric power of $T^*M$.
\begin{definition}[Generic Metrics, \cite{JD}]
By a generic Riemannian metric on a smooth manifold $M$ we mean a Riemannian metric that belongs to a countable intersection of open dense subsets of the spaces of $L^2_k$ Riemannian metrics on $M$ with the $L^2_k$ topology, for some choice of $k\in\mathbb{N}$, $k\geq 2$.
\end{definition}
\begin{remark}
Notice that \textit{generic metric} always implies a countable intersection of open dense subsets of the metric space. For geodesics and harmonic maps the metric space is $L^2_k$ \cite{JD}, and for minimal submanifolds it's $C^{q}$ Riemannian metric for $q\geq 3$, \cite{BW}. 
\end{remark}
\begin{theorem}[Bumpy Metric Theorem for Minimal Submanifold, \cite{BW}]\label{bmt}
If $M$ is a compact manifold, then for a generic choice of metric of $C^q$ $g$ on $M$ ($q\geq 3$), there are no minimal submanifolds with nonzero normal Jacobi fields. That is, each minimal submanifold has nullity $0$.
\end{theorem}

\begin{definition}[Prime Harmonic Map]
Suppose that $f,h:\Sigma\to M$ are harmonic maps. $f$ is called a branched cover of $h$ if there exists a holomorphic map $g:\Sigma\to\Sigma$ of degree $d\geq 2$ such that $f=h\circ g$. We call $f$ a prime harmonic map if it's not a branched cover of another harmonic map.
\end{definition}
\begin{definition}[Branch Point] A point $p\in\Sigma$ is a branch point for the harmonic map $f:\Sigma\to M$ if $(\partial f/\partial z)(p)=0$, where $z$ is any complex coordinate near $p$. If $p$ is a branch point of $f$ but there exists some neighborhood $V$ containing $p$ such that $f(V)$ is an immersed surface, then we say that $p$ is a false branch point.
		\end{definition}
		\begin{definition}[Injective Point]
			If $f:\Sigma\to M$ is a parametrized minimal surface we say that $p\in\Sigma$ is an \textit{injective point} for $f$ if 
			\[df(p)\neq 0,\text{ and}\quad f^{-1}(f(p))=0.\]
			If $f:\Sigma\to M$ is connected and has injective points, we say it is \textit{somewhere injective}.
		\end{definition}
		\begin{lemma} \cite[Lemma 5.2.1]{JD}\label{bp}
			If a conformal harmonic map $f:\Sigma\to M$ is prime, its injective points form an open dense subset of $\Sigma$.
		\end{lemma}
		\begin{theorem}\cite[Theorem 5.1.1]{JD}\label{jdbumpy}
			If $M$ is a compact smooth manifold of dimension at least three, then for a generic choice of Riemannina metric on $M$, all prime compact parametrized minimal surfaces $f:\Sigma\to M$ are free of branch points and lie on nondegenerate critical submanifolds, each such submanifold being an orbit of the group $G$ of conformal automorphisms of $\Sigma$ which are homotopic to the identity. 
		\end{theorem}
\begin{remark}
In Theorem \ref{jdbumpy}, nondegeneracy of a prime confromal harmonic map means that the Jacobi field of it are those generated by the conformal automorphisms of $\Sigma.$
		\end{remark}
\begin{definition}[Bumpy Metric]\label{bumpy metric}
 A metric $g$ on $M$ is called \textit{bumpy} if 
 \begin{enumerate}
     \item there is no smooth immersed minimal submanifold with a non-trivial Jacobi field, \item there's no prime compact parametrized minimal surface with a non-trivial Jacobi field other than the Jacobi fields generated by the conformal automorphisms of its domain,
     \item there's no prime compact parametrized minimal surface with transversal crossings,
     \item the prime compact parametrized minimal surface is free without branch points.
 \end{enumerate}
\end{definition}
In the case $\Sigma=\mathcal{T}_{\tau}$, the conformal automorphism of $\mathcal{T}_{\tau}$ is $S^1\times S^1.$ Given a conformal harmonic torus $u:\mathcal{T}_{\tau}\to M$, if $u$ is prime then it's free of branch points by Theorem \ref{jdbumpy}. If $u$ is not prime it can be written as a branched cover of a prime harmonic torus, and all of its branched points are false. Theorem \ref{jdbumpy} implies that for all conformal harmonic torus $u:T_\tau\to M$, $u(T_\tau)$ is a smooth immersed minimal submanifold. 
\subsection{Geodesic coordinate and geodesic frame}\label{sec_coordinate}

For any $p\in\Sigma$, we can construct a ``partial" geodesic coordinate and a geodesic frame on a neighborhood of $p$ in $M$ as follows:

\begin{enumerate}
\item Choose an oriented, orthonormal basis $\{e_1,e_2\}$ for $T_p\Sigma$.  The map
\begin{align*}
F_0: \mathbf{x} = (x^1,x^2) &\mapsto \exp^\Sigma_p(x^j e_j)
\end{align*}
parametrizes an open neighborhood of $p$ in $\Sigma$, where $\exp^\S$ is the exponential map of the induced metric on $\Sigma$.  For any $\mathbf{x}$ of unit length, the curve $\gamma(t) = F_0(t\mathbf{x})$ is called a \emph{radial geodesic on $\Sigma$} (\emph{at $p$}).  By using $\nabla^\Sigma$ to parallel transport $\{e_1,e_2\}$ along these radial geodesics, we get a local orthonormal frame for $T\Sigma$ on a neighborhood of $p$ in $\Sigma$.  The frame is still denoted by $\{e_1,e_2\}$.
\item Choose an orthonormal basis $\{e_3,\cdots, e_{n}\}$ for $N_p\Sigma$.  By using $\nabla^\bot$ to parallel transport $\{e_3,\cdots,e_{n}\}$ along radial geodesics on $\Sigma$, we obtain a local orthonormal frame for $N\Sigma$ on a neighborhood of $p$ in $\Sigma$.  This frame is still denoted by $\{e_3,\cdots,e_n\}$. It is clear that $\{e_1,e_2, e_3,\cdots,e_n\}$ is a local orthonormal frame for $TM|_\Sigma$.
\item The map
\begin{align*}
F: (\mathbf{x},\mathbf{y}) = \big((x^1,x^2), (y^{3},\cdots,y^{n})\big) &\mapsto \exp_{F_0(\mathbf{x})}(y^\alpha e_{\alpha})
\end{align*}
parametrizes an open neighborhood of $p$ in $M$.  The map $\exp$ is the exponential map of $(M,g)$.  For any $\mathbf{y}$ of unit length, the curve $\Sigma(t) = F(\mathbf{x},t\mathbf{y}) = \exp_{F_0(\mathbf{x})}(t\mathbf{y})$ is called a \emph{normal geodesic for $\Sigma\subset M$}.
\item For any $\mathbf{x}$, step (ii) gives an orthonormal basis $\{e_1,\cdots,e_{n}\}$ for $T_{F(\mathbf{x},0)}M$. By using $\nabla$ to parallel transport it along normal geodesics, we have an orthonormal frame for $TM$ on a neighborhood of $p$ in $M$.  This frame is again denoted by $\{e_1,\cdots,e_{n}\}$.
\end{enumerate}
\subsection{The Space of Conformal Harmonic Tori with Bounded Energy is Countable}
\begin{lemma}\label{bumpyy}
Given a closed smooth manifold $M$ of dimension $n\geq 3$. Suppose $(M,g)$ is a bumpy metric, then for a conformal harmonic torus $u:\mathcal{T}_{\tau}\to M$. There exists $\epsilon>0$, which depends on the equivalent class of $u$, such that for any conformal harmonic torus $f:\mathcal{T}_{\tau'}\to M$, if $|\tau-\tau'|<\epsilon$ and \[\int_{\mathcal{T}_{\tau}}|\nabla (f\circ\iota_{\tau,\tau'})-\nabla u|^2<\epsilon,\] 
then $[f]=[u].$ 
\end{lemma}
\begin{proof}
Given a conformal harmonic map $u:T_\tau\to M$, Theorem \ref{jdbumpy} implies that the image $u(T_\tau)$ is free without branch points, thus a smooth immersed minimal surface.

We argue by contradiction. If not, then there exists a sequence of conformal harmonic torus $\{f_l\}_{l\in\mathbb{N}}$, $f_l:\mathcal{T}_{\tau_l}\to M$, such that $|\tau-\tau_l|\to 0$ and \[\int_{\mathcal{T}_{\tau}}|\nabla (f_l\circ\iota_{\tau,\tau_l})-\nabla u|^2\to 0,\]
and $[f_l]\neq[u]$ for all $l\in\mathbb{N}.$ By strong convergence in $W^{1,2}$, Theorem \ref{regularity}, and Arzel\`a-Ascoli theorem we know that the convergence $f_l\to u$ is smooth and uniform. 
			
First we consider the case that the image of the maps are different, i.e., $f_{l}(\mathcal{T}_{\tau_l})\neq u(T_\tau)$ for all $l\in\mathbb{N}$. Since the convergence is smooth, for $l$ sufficiently large, $f_l(\mathcal{T}_{\tau_l})$ can be written as \textit{graph} of $u(T_\tau)$. In other words, choose $p\in u(T_\tau)$, we can use local partial geodesic coordinate centered at $p$ (see Section \ref{sec_coordinate} for the construction) and write $f_l(\mathcal{T}_{\tau_l})$ as the following
            \[\exp_{F(\mathbf{x},0)}(\eta^l_\alpha(\mathbf{x})e_\alpha)=\phi_l(\mathbf{x}),\]
			for some smooth real functions $\eta^l_\alpha:\Sigma\to\mathbb{R}$. Thus we have
			\begin{equation}\label{722}
			    \begin{split}
		d\phi_l(e_i)=&e_i+\big(\frac{\partial\eta^l_\alpha}{\partial x_i}\big)e_\alpha+\eta^l_\alpha\nabla_{e_i}e_\alpha\\
           =&e_i+\big(\frac{\partial\eta^l_\alpha}{\partial x_i}\big)e_\alpha\\
           &+\eta^l_\alpha\Big(-h_{\alpha ij}|_pe_j-\eta^l_\beta(R_{\alpha i\beta j}+\sum_kh_{\alpha ik}h_{\beta jk})|_p e_j-\eta^l_\gamma R_{\alpha\beta\gamma i}|_p e_\beta\Big)\\
           &+\mathcal{O}(|\eta^l|^3),   
			    \end{split}
			\end{equation}
where the indexes $i,j,k$ are either $1$ or $2$ and the indexes $\alpha, \beta, \gamma$ range from $3$ to $n$. The third line of equation \eqref{722} follows from the following (for the proof see \cite[Proposition 2.6]{MTW}).
\begin{align*}
\nabla_{e_i}e_\alpha|_{\phi_l}=&\big(-h_{\alpha ij}|_p-\eta^l_\beta(R_{\alpha i\beta j}+\sum_kh_{\alpha ik}h_{\beta jk})|_p\big)e_j\\
&-\eta^l_\gamma R_{\alpha\beta\gamma i}|_p e_\beta +\mathcal{O}(|\eta^l|^2).    
\end{align*}
Now we compute the \textit{metric} $\tilde{g}^l_{ij}$ for the \textit{graph} $\phi_l(\mathbf{x})$ 
\begin{align*}
\tilde{g}^l_{ij}=&\langle d\phi_l(e_i),d\phi_l(e_j)\rangle\\
=&\delta_{ij}+\frac{\partial\eta^l_\alpha}{\partial x_i}\frac{\partial\eta^l_\alpha}{\partial x_j}-2\eta^l_\alpha h_{\alpha ij}-2\eta^l_\alpha\eta^l_\beta R_{\alpha i\beta j}-2\eta^l_\alpha\eta^l_\beta h_{\alpha ik}h_{\beta jk}+\mathcal{O}(|\eta^l|^3).
\end{align*}
Since we will consider solutions where $|\eta^l|+|\nabla\eta^l|$ is small. We will often write $\tilde{g}^l_{ij}$ as
\[\tilde{g}_{ij}=\delta_{ij}-2\eta^l_\alpha h_{\alpha ij}+\mathcal{Q}_{ij},\]
where $\mathcal{Q}_{ij}$ denotes a matrix that is quadratic in $\eta^l_\alpha$ and $\nabla\eta^l_\alpha$ and will be allowed to vary from line to line. It follows that the inverse metric is given by
\[(\tilde{g}^l_{ij})^{-1}=\delta_{ij}+2\eta^l_\alpha h_{\alpha ij}+\mathcal{Q}_{ij},\]
then we have 
\[\det \tilde{g}^l_{ij}=(1+2\eta^l_\alpha h_{\alpha ij})(1-2\eta^l_\alpha h_{\alpha ij})+\mathcal{Q}=1+\mathcal{Q},\]
where the equality used minimality of $u(T_\tau)$ and $\mathcal{Q}$ this time is a scalar quadratic term. Whenever $J(s)$ is a differentiable path of matrices, the derivative at $0$ of the determinant is given by 
\[\frac{d}{ds}\at[\Big]{s=0}\det J(s)=\det J(0)\text{Trace}(J^{-1}(0)J'(0)).\]
Applying this with $J(s)=\tilde{g}^l_{ij}(s)$, where $\tilde{g}^l_{ij}(s)$ is computed with $\eta^l_\alpha+s\nu_\alpha$ in place of $\eta^l_\alpha$, let $X_l:=\eta^l_\alpha e_\alpha$ we get
\begin{align*}
\frac{d}{ds}\at[\Big]{s=0}\det J(s)=&(1+\mathcal{Q})(\delta_{ij}+2\eta^l_\alpha h_{\alpha ij}+\mathcal{Q}_{ij})\\
&\Big(2\frac{\partial\nu_\alpha}{\partial x_i}\frac{\partial\eta^l_\alpha}{\partial x_j}-2\nu_\alpha h_{\alpha ij}-2\eta^l_\alpha\nu_\beta R_{\alpha i\beta j}-2\eta^l_\alpha\nu_\beta h_{\alpha jk}h_{\beta ik}+\mathcal{Q}_{ij}\Big)\\
=&2\Big(\frac{\partial\nu_\alpha}{\partial x_i}\frac{\partial\eta^l_\alpha}{\partial x_i}-2\eta^l_\alpha\nu_\beta R_{\alpha i\beta i}-2\eta^l_\alpha\nu_\beta h_{\alpha ik}h_{\beta ik}\Big)+\mathcal{Q}_{ij}\\
=&-2\langle(\Delta_\Sigma^NX_l+\text{Tr}[R_M(\cdot,X_l)\cdot]+\tilde{A}(X_l)),\nu_\alpha e_\alpha\rangle+\mathcal{Q}_{ij},
\end{align*}
where $\tilde{A}$ is the Simons' operator and $\Delta_\Sigma^N$ is the Laplacian on the normal bundle(see \cite[pg.41]{CDD} for the definitions). Since $f_l(\mathcal{T}_{\tau_l})$ is minimal we get that $X_l=\eta^l_\alpha e_\alpha$ satisfies the following equation
\begin{equation}\label{733}
\Delta_\Sigma^NX_l+\text{Tr}[R_M(\cdot,X_l)\cdot]+\tilde{A}(X_l))+\mathcal{Q}_{ij}=0.    
\end{equation}
With the assumption $f_l(\mathcal{T}_{\tau_l})\neq u(T_\tau)$ we can choose $q\in u(T_\tau)$ such that $\eta^l_\beta(q)\neq 0$ for all $l\in\mathbb{N}$ and for some $\beta\in\{3,...,n\}.$ We now consider 
\[\tilde{X}_l(\mathbf{x})=\frac{\eta^l_\alpha(\mathbf{x}) e_\alpha}{\eta^l_\beta(q)},\quad\mathbf{x}\in u(T_\tau),\] 
then $\langle \tilde{X}_l(q),e_\beta\rangle=1$ for all $l$. $\tilde{X}_l(\mathbf{x})$ converges uniformly to a normal vector field $X$ on $\Sigma$. Multiplying equation \eqref{733} by $1/\eta^l_\beta(q)$ and using the uniform convergence we can pass to limits to get that
\[\Delta_\Sigma^NX+\text{Tr}[R_M(\cdot,X)\cdot]+\tilde{A}(X)=0.\]
Thus $X$ is a Jacobi field of $\Sigma$ and it contradicts theorem \ref{bmt}.

Now we consider the case that $f_{l}(\mathcal{T}_{\tau_l})=u(T_\tau)$ for all $l\in\mathbb{N}$. Since the convergence is smooth we can assume that
\begin{equation}\label{souri}
|f_l\circ\iota_{\tau_l,\tau}(x)-u(x)|<1/l,\:\forall x\in T_\tau,    
\end{equation}
and
\begin{equation}
\Big|\frac{\partial f_l\circ\iota_{\tau,\tau_l}(x)}{
\partial x^{\alpha}}-\frac{\partial u(x)}{\partial x^{\alpha}}\Big|<1/l,\:\forall x\in T_\tau.
\end{equation}
We define $w$ as the following 
\begin{equation}
    w(x):=\lim_{l\to\infty}\frac{f_{l}\circ\iota_{\tau,\tau_l}(x)-u(x)}{\max_{y\in \mathcal{T}_{\tau}}|f_{l}\circ\iota_{\tau,\tau_l}(y)-u(y)|},\quad\forall x\in T_\tau,
\end{equation}
If $u$ is a prime harmonic sphere, then Proposition \ref{prop7} implies that $w$ is a tangential Jacobi field of the map $u$. Theorem \ref{jdbumpy} implies that $W$ generated by the conformal automorphism of $T_\tau$, contradicting the assumption $[f_l]\neq[u]$. Now we consider the case that $u$ is not prime. \begin{claim}
For $l$ sufficiently large, we have
\[deg(f_l)=deg(u).\]
\end{claim}
\begin{proof}[Proof of the claim]
It's known that if two maps are homotopic then they have the same degree. The stradegy of the proof is similar and can be found in \cite[Chapter 5]{dt}.
				
For $l$ sufficiently large, by equation \eqref{souri} we know that $f_{l}\circ\iota_{\tau,\tau_l}$ is homotopic to $u$. That is, there exists a continuous map $H_l:T_\tau\times[0,1]\to u(T_\tau)$:
\[H_l(x,t):=\gamma_x^l(t),\:x\in T_\tau,\]
such that $H_l(x,0)=u(x)$ and $H_l(x,1)=f_{l}\circ\iota_{\tau,\tau_l}(x)$ for all $x\in T_\tau$, where $\gamma_x^j(t)$ denotes the unique geodesic with respect to the intrinsic metric starting from $u(x)$ with end point at $f_{l}\circ\iota_{\tau,\tau_l}(x)$. Since $f_{l}(\mathcal{T}_{\tau_l})=u(T_\tau)$, and by Lemma \ref{bp} we can choose $y\in u(T_\tau)$ so it satisfies the following conditions:
\begin{enumerate}
\item $\forall p\in u^{-1}(y)$, $du(p)\neq 0$,
\item $\forall p\in\big(f_l\circ\iota_{\tau,\tau_l}\big)^{-1}(y)$, $d\big(f_l\circ\iota_{\tau,\tau_l}\big)(p)\neq 0$,
\item $y$ is a regular value for $H_l$ and $H_l|_{\partial(T_\tau\times[0,1])}$.
\end{enumerate}
$H_l^{-1}(y)$ is a compact 1-dimensional submanifold with boundary $(H_j|_{\partial(T_\tau\times[0,1])})^{-1}(y)$. In other words, it contains embedded arcs which are transverse to $\partial(T_\tau\times[0,1])$. By \cite[Chapter 5]{dt}, given $p_1\in H_l^{-1}(y)$, there is an unique $p_2\in H_l^{-1}(y)$, $p_1\neq p_2$, and a component arc $\Gamma\in H_l^{-1}(y)$ with $\partial\Gamma=\{p_1,p_2\}$. By \cite[Chapter 5, Lemma 1.2]{dt}, $p_1$ and $p_2$ are of opposite type for $H_l|_{\partial(T_\tau\times[0,1])}$. This implies that for any $p_1\in u^{-1}(y)$, there is an unique corresponding $p_2\in \big(f_l\circ\iota_{\tau,\tau_l}\big)^{-1}(y)$. Then it follows that degree of $u$ and degree of $f_l$ are the same for $l$ sufficiently large.  \end{proof}
We can write $u=\tilde{u}\circ\phi$ for some conformal prime harmonic map $\tilde{u}$ and a holomorphic map $\phi:T_\tau\to T_\tau$ with degree $deg(u)$. By the smooth convergence of $f_l\circ\iota_{\tau,\tau_l}\to u$, and $deg (f_{l})=deg(u)$, we can write $f_{l}$ as a branched cover of some prime harmonic map $g_{l}:\mathcal{T}_{\tau_l}\to M$, i.e., \[f_{l}=g_{l}\circ\big(\iota_{\tau,\tau_l}\circ\phi\circ\iota_{\tau_l,\tau}\big).\] 
So we obtain a sequence of conformal prime harmonic tori $g_l:\mathcal{T}_{\tau_l}\to M$ that converges strongly to a conformal prime harmonic map $\tilde{u}$ in $W^{1,2}$, and $[f_l]\neq[u]$ implies that $[g_l]\neq[\tilde{u}]$. It is the desired contradiction.
\end{proof}	
We define $\mathcal{F}^W$ to be the equivalent classes of all conformal harmonic tori with energy bounded by $W$, that is:	$$\mathcal{F}^W:=\Big\{[f]\mid f:\mathcal{T}_{\tau}\rightarrow M\text{ conformal and harmonic}, E(f,T_\tau)\leq W\Big\}.$$
\begin{proposition}\label{countable}
The set $\mathcal{F}^W$ is countable.
\end{proposition}
\begin{proof}
For $B_R\subset\mathbb{C}$, we pick a finite set $\{x^1_n,...,x^{p^n}_n\}$ of $\mathbb{C}$ such that $\bar{B}_R\subset\bigcup_{k=1}^{p^n} B_{\frac{1}{n}}(x^k_n).$ We define $\mathcal{F}_R^W(n)$ to be:
\begin{equation}
    \begin{split}
    \mathcal{F}_R^W(n)=\Big\{[f]\mid f\in\mathcal{F}^W&,\:f:\mathcal{T}_{\tau}\to M,\:\{1,\tau\}\in B_R\\
&\int_{B_{\frac{1}{n}}(x^k_n)}|\nabla f|^2<\epsilon_{su},\:\text{for all  }k=1,...,p^n\Big\},        
    \end{split}
\end{equation}
where $\{1,\tau\}\in B_R$ means that $\bar{B}_R$ contains all parallelograms generates by $\{1,\tau\}$, and $\epsilon_{su}>0$ is the constant given in Theorem \ref{regularity}. We can see that:
$$\mathcal{F}^W=\bigcup_{n\in\mathbb{N},R\in\mathbb{R}}\mathcal{F}_R^W(n).$$ 
Now we prove that $\mathcal{F}^W(n)$ is finite. If not, then there exists a sequence of conformal harmonic tori $\{[f_l]\}_{l\in\mathbb{N}}\subset \mathcal{F}_R^W(n)$, $f_l:\mathcal{T}_{\tau_l}\to M$, with $[f_i]\neq[f_j]$. Because $\bar{B}_R$ is compact we have that $\tau_l\to\tau$ up to subsequence and $f_l$ converges to a conformal harmonic map $f:T_\tau\to M$, up to subsequence. The assumption $\{[f_l]\}_{l\in\mathbb{N}}\subset \mathcal{F}_R^W(n)$ implies that
$$\int_{B_{1/n}(x^k_n)}|\nabla f_l|^2<\epsilon_{su}\quad\forall l\in\mathbb{N},k=1,...,p^n.$$ 
By Theorem \ref{regularity}, this implies that the convergence is strong in $W^{1,2}$ and $f\in\mathcal{F}^W(n)$. Then it contradicts Lemma \ref{bumpyy}.
		\end{proof}	

\bibliographystyle{alpha}
\bibliography{main.bib}	
\end{document}